\documentclass[11pt,           
draft,                         
english                       
]{article}

\usepackage[english,strings]{babel}     
\usepackage[utf8]{inputenc}    
\usepackage[T1]{fontenc}       
\usepackage{longtable}         
\usepackage{exscale}           
\usepackage[final]{graphicx}   
\usepackage[sort]{cite}        
\usepackage{array}             
\usepackage{fancyhdr}          
\usepackage[a4paper]{geometry} 
\usepackage{xspace}            
\usepackage{tikz}              
\usepackage{tikz-cd}           
\usepackage{ifdraft}           
\usepackage{nchairx}            
\usepackage{eucal}            
\usepackage[expansion=false    
           ]{microtype}        
\usepackage[nottoc]{tocbibind} 
\usepackage[backref=page,      
           final=true,         
           pdfpagelabels       
           ]{hyperref}         

\usepackage{faktor}  
\usepackage{xpatch}
\graphicspath{{../tikz/}}

\usetikzlibrary{matrix,calc,fadings,decorations.pathreplacing,positioning,patterns}

\ifdraft{\synctex=1}{}

\newcommand{\myname}{\textbf{Stefan Waldmann}}
\newcommand{\myemail}{\texttt{stefan.waldmann@mathematik.uni-wuerzburg.de}}
\newcommand{\myaddress}{Julius Maximilian University of Würzburg \\
	Department of Mathematics \\
	Chair of Mathematics X (Mathematical Physics) \\
	Emil-Fischer-Straße 31 \\
	97074 Würzburg \\
	Germany}

\newcommand{\AuthorOne}{\textbf{Marvin Dippell}}
\newcommand{\AuthorTwo}{\textbf{Chiara Esposito}}
\newcommand{\AuthorThree}{\myname}

\author{\AuthorOne\thanks{\AuthorEmailOne},\\[0.2cm]
	\AuthorAddressOne
	\\[0.5cm]
	\AuthorTwo\thanks{\AuthorEmailTwo},\\[0.2cm]
	\AuthorAddressTwo
	\\[0.5cm]
	\AuthorThree\thanks{\AuthorEmailThree}\\[0.2cm]
	\AuthorAddressThree
	\\[0.5cm]
}

\newcommand{\AuthorAddressOne}{\myaddress}
\newcommand{\AuthorAddressTwo}{Dipartimento di Matematica\\
	Università degli Studi di Salerno\\
	via Giovanni Paolo II, 123\\
	84084 Fisciano (SA)\\
	Italy}
\newcommand{\AuthorAddressThree}{\myaddress}

\newcommand{\AuthorEmailOne}{\texttt{marvin.dippell@mathematik.uni-wuerzburg.de}}
\newcommand{\AuthorEmailTwo}{\texttt{chesposito@unisa.it}}
\newcommand{\AuthorEmailThree}{\myemail}

\tikzset{triplearrow/.style={
		draw=black!75,
		color=black!75,
		double distance=4pt,
		-
	},
	thirdline/.style={
		draw=black!75,
		color=black!75,
		-{>[scale=2.0]}
}}

\makeatletter
\newcommand{\xRrightarrow}[2][]{\ext@arrow 0359\Rrightarrowfill@{#1}{#2}}
\newcommand{\Rrightarrowfill@}{\arrowfill@\equiv\equiv\Rrightarrow}
\makeatother
\newcommand{\longRrightarrow}{\xRrightarrow{\,\quad}}

\newcommand{\deform}[1]{\boldsymbol{#1}}

\newcommand{\Null}{{\scriptscriptstyle0}}
\newcommand{\NULL}{0}
\newcommand{\Wobs}{{\scriptscriptstyle\normalizer}}
\newcommand{\WOBS}{\normalizer}
\newcommand{\Total}{{\scriptscriptstyle\mathrm{tot}}}
\newcommand{\TOTAL}{\mathrm{tot}}
\newcommand{\Trivial}{{\scriptstyle\mathrm{trivial}}}
\newcommand{\Unreduce}{{\scriptstyle\mathrm{unred}}}
\renewcommand{\red}{{\scriptscriptstyle\mathrm{red}}}

\newcommand{\qalgebra}[1]{\pmb{\algebra{#1}}}
\newcommand{\qmodule}[1]{\pmb{\module{#1}}}

\newcommand{\algbimodule}[3]{\deco{}{\algebra{#1}}{\module{#2}}{}{\algebra{#3}}}
\newcommand{\qalgbimodule}[3]{\deco{}{\qalgebra{#1}}{\qmodule{#2}}{}{\qalgebra{#3}}}

\newcommand{\normalizer}{\mathrm{N}}

\newcommand{\reduce}{\functor{red}}
\renewcommand{\cl}{\functor{cl}}

\newcommand{\lidentity}{\operator{left}}
\newcommand{\ridentity}{\operator{right}}
\newcommand{\Id}{\operator{Id}}

\newcommand{\bicategory}[1]{\category{#1}}

\newcommand{\CoisoAlgTriple}{\mathsf{C_3Alg}}

\newcommand{\LeftIdeals}{\mathsf{LeftIdealAlg}}

\newcommand{\CoisoAlgPair}{\mathsf{C_2Alg}}

\newcommand{\CoisoBimodTriple}{\mathsf{C_3Bimod}}

\newcommand{\CoisoBimodPair}{\mathsf{C_2Bimod}}

\newcommand{\CoisoBimod}{\mathsf{CoisoBimod}}

\newcommand{\Bimodule}{\categoryname{Bimodule}}

\newcommand{\Pic}{\mathsf{Pic}}

\title{Coisotropic Triples, Reduction and Classical Limit}

\date{\today}

\begin{document}

\selectlanguage{english}

%
%

\maketitle

%
%

\begin{abstract}
    Coisotropic reduction from Poisson geometry and deformation
    quantization is cast into a general and unifying algebraic
    framework: we introduce the notion of coisotropic triples of
    algebras for which a reduction can be defined. This allows to
    construct also a notion of bimodules for such triples leading to
    bicategories of bimodules for which we have a reduction functor as
    well. Morita equivalence of coisotropic triples of algebras is
    defined as isomorphism in the ambient bicategory and characterized
    explicitly. Finally, we investigate the classical limit of
    coisotropic triples of algebras and their bimodules and show that
    classical limit commutes with reduction in the bicategory sense.
\end{abstract}

%
%

\tableofcontents
\newpage

%
%

\section{Introduction}
\label{sec:Introduction}

Coisotropic reduction is one of the standard constructions in Poisson
geometry leading to a new reduced Poisson manifold obtained out of the
given data of a Poisson manifold with a coisotropic submanifold. Of
course, geometrically certain circumstances have to be met in order to
obtain a smooth reduced Poisson manifold. Ignoring these geometric
assumptions, an algebraic formulation of coisotropic reduction is
possible and works in general, yielding a reduced Poisson algebra out
of a given Poisson algebra with a coisotropic ideal: here an
associative ideal in a Poisson algebra is called coisotropic if it is
a Poisson subalgebra, though not necessarily a Poisson ideal.

The original motivation to consider coisotropic submanifolds and the
corresponding reduction comes from Dirac's program \cite{dirac:1950a,
  dirac:1964a} to handle constraint mechanical systems: the notion of
a coisotropic submanifold corresponds to the first-class
constraints. Dirac's intention ultimately was of course to pass to a
quantum theory. This leads to the task to find a quantized version of
coisotropic reduction as well.

Among many approaches one can favour deformation quantization
\cite{bayen.et.al:1978a} as starting point. Here various versions of
phase space reduction are available, starting with a BRST approach in
\cite{bordemann.herbig.waldmann:2000a} and more general coisotropic
reduction schemes found in e.g. \cite{cattaneo.felder:2007a,
  cattaneo:2004a, cattaneo.felder:2004a,
  bordemann.herbig.pflaum:2007a, bordemann:2005a,
  bordemann:2004a:pre}. The general idea is that the functions
vanishing on the coisotropic submanifold should be deformed into a
\emph{left ideal} of the ambient algebra of all functions. The reduced
algebra is then the quotient of the normalizer of this left ideal
modulo the left ideal. Also, it is worthwhile to mention that
the quantization of coisotropic subgroups has been considered
in the context of quantum groups, see
e.g \cite{ciccoli:1997a, ciccoli.gavarini:2006a}.

Since both versions, the classical reduction as well as the quantum
reduction, can be formulated entirely in an algebraic fashion it is
reasonable to explore the algebraic features further, independent of
the possible geometric origin. Then one important question is how the
relations between different algebras with coisotropic ideals behave
after reduction. A standard question beyond the behaviour of
isomorphisms is then the behaviour of Morita equivalence.

Thus the first main question we want to address is how Morita
equivalences between reduced algebras can be encoded in the data
before reduction.

The main idea to approach this is to put Morita theory in a slightly
larger context of an appropriate bicategory: for algebras (or rings)
it is a well-known procedure that the bicategory of all bimodules
$\Bimodules$ encodes Morita equivalence as the notion of isomorphism in
the bicategorical sense. However, now one has much more structure as
also bimodules between algebras enter the game which are not
necessarily equivalence bimodules: they can carry important
information themselves.

Thus our first step is to construct a bicategory for the situation
before reduction which allows for a functorial reduction. It turns out
that the first idea of algebras $\algebra{A}_\Total$ with a specified
left ideal $\algebra{A}_\Null$ are not yet the suitable notion of
objects in this bicategory. One simply can not define a reasonable
notion of bimodules over such pairs that is compatible with reduction.
Thus our approach consists in taking triples of an ambient algebra:
the total algebra $\algebra{A}_\Total$, a subalgebra $\algebra{A}_\Wobs$ of weakly
observables in Dirac's original sense, and a two-sided ideal $\algebra{A}_0$ in this
subalgebra, corresponding to the left ideal from before. The idea in
mind is that starting with an algebra with left ideal one has to add
the normalizer of this left ideal as algebra in the
middle. Nevertheless, we give classes of interesting examples where
one needs additional freedom to choose this algebra in the middle,
thus justifying to consider what we call \emph{coisotropic triples} of
algebras
$\algebra{A} = (\algebra{A}_\Total, \algebra{A}_\Wobs,
\algebra{A}_\Null)$
in the following.  We then define a notion of bimodules over such
triples allowing for a good tensor product: this ultimately leads us
to the \emph{bicategory of coisotropic triples} $\CoisoBimodTriple$
where the 1-morphisms are bimodules over coisotropic triples of
algebras as objects with an appropriate notion of bimodule morphisms
as 2-morphisms.

Having this bicategory, it allows now to speak of Morita equivalence
of coisotropic triples of algebras which, by definition, is
isomorphism in the sense of $\CoisoBimodTriple$. As a first result we
give an explicit characterization of Morita equivalent triples in
Theorem~\ref{theorem:MEBimoduleStructure}: it implies Morita
equivalence of the corresponding $\TOTAL$- and $\WOBS$-components
together with a compatibility condition between the three components
of the triples.

The second step consists now in extending the reduction of algebras to
bimodules. We show in Theorem~\ref{theorem:ReductionBicCoisoBimod}
that this is indeed possible and leads to a reduction functor
\begin{equation*}
    \reduce\colon \CoisoBimodTriple \longrightarrow \Bimodules
\end{equation*}
in the bicategory sense.  Since we have an honest functor between
bicategories, the reduction maps isomorphic objects to isomorphic
objects and hence preserves Morita equivalence being the notion of
isomorphism in $\CoisoBimodTriple$ and $\Bimodules$, respectively.
However, being a functor we get much more detailed information from
reduction, say a bigroupoid morphism of the corresponding Picard
bigroupoids, i.e. the bigroupoids of isomorphisms in these
bicategories.

For technical reasons it is convenient to consider only the components
$(\algebra{A}_\Wobs, \algebra{A}_\Null)$ of a coisotropic triple of
algebras leading to the notion of a coisotropic pair of algebras: the
reduction uses only this information. Now the construction of
$\CoisoBimodTriple$ can be adjusted to yield also a bicategory
$\CoisoBimodPair$ of bimodules over such pairs together with the
corresponding reduction functor
\begin{equation*}
    \reduce\colon \CoisoBimodPair \longrightarrow \Bimodules.
\end{equation*}

As we have seen, one of the main motivations to consider coisotropic
reduction is to pass from a classical to a quantum system and use the
classical data to investigate the reduced quantum system. Thus, in a
last step, we consider general deformations of coisotropic triples of
algebras and their bimodules. While a \emph{quantization} of bimodules
is typically obstructed, not unique, and fairly difficult to
understand in general, the \emph{classical limit} is always rather
easy to study and unobstructed. We define a classical limit of
coisotropic triples of algebras over a ring $\ring{R}\formal{\lambda}$
of formal power series in a formal parameter $\lambda$ with
coefficients in a ring $\ring{R}$ as the quotient by the ideals of
multiples of $\lambda$. The idea is that the algebras over
$\ring{R}\formal{\lambda}$ are interpreted as deformations of algebras
over $\ring{R}$. While the classical limit of the algebras is rather
straightforward, we then are able to extend the classical limit also
to bimodules leading to a functor
\begin{equation*}
    \cl\colon
    \CoisoBimodTriple_{\ring{R}\formal{\lambda}}
    \longrightarrow
    \CoisoBimodTriple_{\ring{R}}
\end{equation*}
of bicategories where now we explicitly indicate the underlying ring
of scalars. As before, we know that the classical limit preserves
Morita equivalence and yields a bigroupoid morphism between the Picard
bigroupoids. In
e.g. \cite{bursztyn.waldmann:2004a,bursztyn.waldmann:2005b} it was
demonstrated that a similar classical limit can be successfully used
to determine the Picard groups of deformed algebras and thus their
Morita theory.

The final result is now that the two functors $\reduce$ and $\cl$
commute in the sense of functors between bicategories: we explicitly
construct the relevant natural transformations and modifications in
Theorem~\ref{theorem:ClassicalLimitCommutesWithReduction} to obtain a
commutative diagram
\begin{equation*}
    \begin{tikzcd}
        \CoisoBimodPair_{\ring{R}[[\lambda]]}
        \arrow{r}{\cl}
        \arrow{d}[swap]{\reduce}
        & \CoisoBimodPair_\ring{R}
        \arrow{d}{\reduce} \\
        \Bimodules_{\ring{R}[[\lambda]]}
        \arrow{r}{\cl}
        & \Bimodules_\ring{R}
    \end{tikzcd}
\end{equation*}
of functors between bicategories. Here it suffices to restrict to
coisotropic pairs instead of triples since the reduction only uses the
information of pairs anyway. In particular, the functors restrict to
commuting bigroupoid morphisms for the corresponding Picard
bigroupoids thus encoding the behaviour of Morita equivalence under
classical limit and reduction completely.

After arriving at this conceptually clear and fairly general picture
of how coisotropic reduction extends to bimodules and relates to the
classical limit several questions arise. We do not address their
solutions in this work but come back to them later on.
\begin{enumerate}
\item The question of \emph{quantization} of coisotropic algebras and
    their modules now becomes more urgent, once having understood
    their classical limit. Here a question of particular interest is
    to understand the quantization of equivalence bimodules provided
    the quantization of the algebras is given.  One can then use
    commutativity of reduction and classical limit to actually find a
    good classification of coisotropic triples of e.g. star product
    algebras in geometric terms like the characteristic classes of the
    underlying star products. This should eventually lead to a
    comparison of the Morita classification of equivariant star
    products initiated in \cite{jansen.waldmann:2006a,
      jansen.neumaier.schaumann.waldmann:2012a}, see also
    \cite{reichert:2017a, reichert.waldmann:2016a,
      esposito.dekleijn.schnitzer:2018a:pre}, extending the Morita
    classification from \cite{bursztyn.dolgushev.waldmann:2012a,
      bursztyn.waldmann:2002a, bursztyn.waldmann:2004a}. On the
    classical level, coisotropic relations provide particular
    coisotropic triples which can then also be taken as starting point
    for quantization \cite{ciccoli:2019a:privat}.
\item The geometric nature of the description of the equivalence
    bimodules from Theorem~\ref{theorem:MEBimoduleStructure} has to be
    clarified further. Here the case of star products is again the
    guiding example and raises the questions what the
    \emph{semi-classical} limit is: the first order structures of
    equivalence bimodules should give analogs of covariant (or
    contravariant) derivatives, now adapted to the coisotropic triple
    point of view analogously to the ordinary case
    \cite{bursztyn.waldmann:2004a, bursztyn:2002a}.
\item Already on the classical level one can try to incorporate the
    first order structures, i.e. the Poisson structures, into the game
    at a more fundamental level. Then a more geometric approach to
    Morita theory in this context could ultimately lead to a notion of
    Morita equivalence of coisotropic triples in Poisson geometry
    yielding the usual Morita equivalence for the reduced Poisson
    manifolds, see \cite{xu:1991a}. Then the question of the behaviour
    of Picard groups as studied in \cite{bursztyn.fernandes:2015a:pre,
      bursztyn.weinstein:2004a} under reduction would be one of the
    first tasks.
\item We know that the reduction functor maps equivalence bimodules
    between the triples to usual equivalence bimodules between the
    reduced algebras. Which other classes of (bi-)modules behave well
    under reduction? Here we want to find suitable criteria to obtain
    e.g. projective modules etc. making contact to the geometric
    framework of reducing vector bundles.
\item A final longer term project is to incorporate $^*$-involutions
    into the definition of coisotropic triples. For applications in
    mathematical physics this is of course crucial but requires some
    severe changes: the main obstacle is that there is no reasonable
    compatibility to require between a left ideal and a
    $^*$-involution. The naive compatibility that the left ideal is
    closed under the involution yields immediately a two-sided ideal
    which in the examples of deformation quantization is known to be
    not relevant at all. One idea would be to require the existence of
    a positive functional having the left ideal as Gel'fand ideal as
    this was done in \cite{gutt.waldmann:2010a} and induce a
    $^*$-involution for the reduced algebra this way. Nevertheless, at
    the moment it seems to be quite unclear how to incorporate the
    corresponding structures like algebra-valued inner products on the
    modules as in \cite{bursztyn.waldmann:2005b,
      bursztyn.waldmann:2001a}. Ultimately, one would like to have a
    definition for \emph{strong Morita equivalence} of coisotropic
    triples of algebras.
\end{enumerate}

The paper is organized as follows: in
Section~\ref{sec:RedcutionPoissonGeometryDQ} we recall the basic
constructions of coisotropic reduction together with some principal
examples. Section~\ref{sec:CoisotropicTriplesPairsAlgebras} contains
the definition of coisotropic triples and pairs of algebras together
with some first functorial properties. The bicategories of bimodules
over coisotropic triples and pairs are constructed in
Section~\ref{sec:TriplePairBimodules} while
Section~\ref{sec:MoritaEquivalence} contains the characterization of
Morita equivalence bimodules. The reduction functor for bimodules is
constructed in Section~\ref{sec:ReductionBimodules}. Finally,
Section~\ref{sec:ClassicalLimit} contains the definition of the
classical limit functor together with the proof of our main result
that classical limit commutes with reduction. In a small appendix we
collect the basic definitions of bicategories, functors, natural
transformations, and modifications as unfortunately there are several
competing versions in the literature: we want to make clear which
definitions we are actually using.

\noindent
\textbf{Acknowledgements:} It is a pleasure to thank Martin Bordemann,
Henrique Burzstyn, Nicola Ciccoli and Rui Loja Fernandes for important
remarks and ideas on this project.

%
%

\section{Preliminaries}
\label{sec:RedcutionPoissonGeometryDQ}

In this section we recall some well-known reduction constructions in
the settings of Poisson geometry and deformation quantization to fix
our notation.

Let $(M, \pi)$ be a Poisson manifold, which models the phase space of
a classical mechanical system, and assume that
$\iota\colon C \longrightarrow M$ is a closed coisotropic submanifold
of $M$, called the \emph{constraint surface}. We denote the vanishing
ideal of $C$ by
\begin{equation}
    \label{eq:VanishingIdealC}
    \algebra{I}_C
    =
    \left\{
        f \in \Cinfty(M)
        \; \big\vert \;
        \iota^* f = 0
    \right\},
\end{equation}
which is clearly an ideal in the associative commutative algebra
$\Cinfty(M)$ of smooth functions on $M$.
\begin{lemma}
    \label{lemma:CoisotropicEquivalences}%
    Let $(M, \pi)$ be a Poisson manifold and $C\subseteq M$ be a
    submanifold.  Then the following statements are equivalent:
    \begin{lemmalist}
    \item The submanifold $C$ is coisotropic.
    \item For a function $f \in \algebra{I}_C$ the Hamiltonian vector
        field $X_f = (\D f)^\sharp \in \Secinfty(TM)$ is tangent to
        $C$, i.e. $X_f(p) \in T_pC \subseteq T_pM$ for all $p \in C$.
    \item Its vanishing ideal $\algebra{I}_C$ is a Poisson subalgebra
        of $\Cinfty(M)$.
    \end{lemmalist}
\end{lemma}
Note that in most interesting situations, $\algebra{I}_C$ will not be
a Poisson ideal: this is equivalent to the statement that $C$ is even
a Poisson submanifold, a situation which is rarely of interest in the
context of reduction.

The distribution on $C$ spanned by the Hamiltonian vector fields $X_f$
of functions $f \in \algebra{I}_C$ is called the \emph{characteristic
  distribution} of $C$.  It turns out that this distribution is
integrable and, under suitable circumstances, has a nice leaf space
$C \big/ \mathord{\sim}$. For simplicity, we assume that the leaf
space is a smooth manifold and the projection onto the leaf space
\begin{equation}
    \label{eq:ProjectionOntoReducedSpace}
    \pr\colon C
    \longrightarrow
    \faktor{C}{\mathord{\sim}}
    =: M_\red
\end{equation}
is a surjective submersion.  In this case, $M_\red$ is itself a
Poisson manifold with Poisson structure determined as follows: one can
characterize the functions on $C$ which are constant along the leaves
as restrictions of functions $f \in \Cinfty(M)$ with the property that
$\{f, g\} \in \algebra{I}_C$ for all $g \in \algebra{I}_C$, i.e. as
the Lie normalizer (or Lie idealizer) of $\algebra{I}_C$ inside
$\Cinfty(M)$. We denote this normalizer as
\begin{equation}
    \label{eq:LieNormalizer}
    \algebra{B}_C
    =
    \left\{
        f \in \Cinfty(M)
        \; \big\vert \;
        \iota^*(X_g f) = 0
        \textrm{ for all }
        g \in \algebra{I}_C
    \right\}.
\end{equation}
It is now an easy verification that $\algebra{B}_C$ is a Poisson
subalgebra of all functions and
$\algebra{I}_C \subseteq \algebra{B}_C$ is a Poisson ideal in its
normalizer. Thus, as an immediate consequence we obtain the following
claim.
\begin{lemma}
    \label{lemma:BcJCisPoisson}%
    Let $(M, \pi)$ be a Poisson manifold and $C\subseteq M$ be a
    coisotropic submanifold.  Then the quotient
    $\algebra{B}_C \big/ \algebra{I}_C$ is a Poisson algebra.
\end{lemma}

Finally, we can observe that the pull-back with the projection yields
an isomorphism
\begin{equation}
    \label{eq:IsoCinftyMredBCIC}
    \pr^*\colon
    \Cinfty(M_\red)
    \longrightarrow
    \faktor{\algebra{B}_C}{\algebra{I}_C}
\end{equation}
of associative algebras. Since the right hand side is a Poisson
algebra in a natural way, this induces the Poisson structure
$\pi_\red$ on the reduced space $M_\red$ whenever $M_\red$ is
a manifold at all with $\pr$ being a surjective submersion. In this
case, the isomorphism turns it into a Poisson manifold as claimed.
But even if this geometric assumption is not satisfied, we can take
$\algebra{B}_C \big/ \algebra{I}_C$ as a valid replacement for
$M_\red$.

\begin{example}[Marsden-Weinstein reduction, classical]
    \label{example:MarsdenWeinstein}%
    A particular but important case of the above procedure is the
    Marsden-Weinstein reduction. Here one assumes to have a smooth
    action $\Phi\colon G \times M \longrightarrow M$ of a connected
    Lie group such that the Poisson structure $\pi$ is preserved.
    Moreover, one requires an $\ad^*$-equivariant momentum map
    $J\colon M \longrightarrow \liealg{g}^*$ where $\liealg{g}^*$ is
    the dual of the Lie algebra $\liealg{g}$ of $G$, i.e. for all
    $\xi \in \liealg{g}$ the fundamental vector field
    $\xi_M \in \Secinfty(TM)$ is given by
    $\xi_M = - X_{J_\xi} = \{J_\xi, \argument\}$ where we define the
    scalar function $J_\xi \in \Cinfty(M)$ as the function obtained by
    pairing the result of $J$ with $\xi$. The equivariance then reads
    as $\{J_\xi, J_\eta\} = J_{[\xi, \eta]}$ for all
    $\xi, \eta \in \liealg{g}$. Equivalently, $J$ is a Poisson map
    with respect to the linear Poisson structure on
    $\liealg{g}^*$. Now one considers the zero level set
    $C = J^{-1}(\{0\})$ of $J$, provided 0 is a regular value and
    $C \ne \emptyset$. Then $C$ is indeed coisotropic and the
    foliation of $C$ is just the foliation by orbits of $G$. Hence in
    this case
    \begin{equation}
        \label{eq:BCforMarsdenWeinstein}
        \algebra{B}_C
        =
        \left\{
            f \in \Cinfty(M)
            \; \big\vert \;
            \iota^* f
            \textrm{ is $G$-invariant}
        \right\}.
    \end{equation}
    Moreover, $M_\red = C \big/ G$, provided the action of $G$ on $C$
    is sufficiently nice: here we assume that $G$ acts freely and
    properly on $C$ so that $\pr\colon C \longrightarrow M_\red$
    becomes a $G$-principal fiber bundle. There are of course many
    generalizations of this particularly simple situation allowing for
    less restrictive assumptions, see e.g. the textbooks
    \cite{ortega.ratiu:2004a, marsden.ratiu:1999a} for further
    information.
\end{example}
\begin{example}
    \label{example:PoissonLieExample}%
    Another example comes from the setting of actions of a Poisson Lie
    group $(G, \pi_G)$ on a Poisson manifold $(M, \pi)$.  One assumes
    to have a smooth action $\Phi\colon G \times M \longrightarrow M$
    of a Poisson Lie group that sends the Poisson structure
    $\pi_G \oplus \pi$ into $\pi$.  In this case a momentum map, if it
    exists, is a map $J\colon M \longrightarrow G^*$ where $G^*$ is
    the dual of the Poisson Lie group $(G, \pi_G)$. Its definition has
    been introduced in \cite{lu:1990a}.  Assuming that $J$ is a
    Poisson map, one can easily see that for any dressing orbit
    $\mathcal{O}_\mu$ its preimage $C := J^{-1}(\{\mathcal{O}_\mu\})$
    by $J$ is a coisotropic submanifold of $M$.  Thus in a similar way
    as the case discussed above, we can obtain a reduced Poisson
    manifold. For further details see \cite{esposito:2014a}.
    Furthermore, the relation between coisotropic submanifolds and
    left ideals in this setting has been proposed in \cite{lu:1993a}.
\end{example}

As a next step we want to incorporate the quantum picture as
well. Here we choose the approach of deformation quantization
\cite{bayen.et.al:1978a}, see e.g. \cite{waldmann:2007a} for an
introduction. Thus we assume to have a formal star product $\star$
given on $(M, \pi)$, i.e. a $\mathbb{C}[[\lambda]]$-bilinear
associative product for $\Cinfty(M)[[\lambda]]$ written as
\begin{equation}
    \label{eq:StarProduct}
    f \star g
    =
    \sum_{r=0}^\infty \lambda^r C_r(f, g)
\end{equation}
with bilinear operators
$C_r\colon \Cinfty(M) \times \Cinfty(M) \longrightarrow \Cinfty(M)$
extended $\mathbb{C}[[\lambda]]$-bilinearly as usual, such that
\begin{equation}
    \label{eq:StarProductLowestOrders}
    C_0(f, g) = fg
    \quad
    \textrm{and}
    \quad
    C_1(f, g) - C_1(g, f) = \I\{f, g\}
\end{equation}
for all $f, g \in \Cinfty(M)$, the constant function $1$ is the unit
of $\star$, and $C_r$ is bidifferential for all $r \in \mathbb{N}_0$.
The resulting algebra is denoted by
$\qalgebra{A} = (\Cinfty(M)[[\lambda]], \star)$. The formal parameter
$\lambda$ corresponds in convergent situations to the physical Planck
constant $\hbar$.

To formulate a quantum version of reduction, we first need to
introduce a quantum analog of the ideal and Poisson subalgebra
$\algebra{I}_C$. One way consists in requiring the existence of a
deformation of $\iota^*$ into a quantum restriction map
\begin{equation}
    \label{eq:QuantumRestrictionMap}
    \deform{\iota^*}
    =
    \iota^* \circ S
    \quad
    \textrm{with}
    \quad
    S = \id + \sum_{r=1}^\infty \lambda^r S_r,
\end{equation}
where $S_r\colon \Cinfty(M) \longrightarrow \Cinfty(M)$ are
differential operators to be found in such a way that
\begin{equation}
    \label{eq:KernelQuantumRestriction}
    \qalgebra{I}_C
    =
    \ker \deform{\iota^*}
    =
    \left\{
        f \in \Cinfty(M)[[\lambda]]
        \; \big\vert \;
        \deform{\iota^*} f = 0
    \right\}
\end{equation}
becomes a left ideal with respect to $\star$. As before, we can then
consider the functions on the constraint surface $C$ and find that
\begin{equation}
    \label{eq:FunctionsOnC}
    \deform{\iota^*}\colon
    \faktor{\Cinfty(M)[[\lambda]]}{\qalgebra{I}_C}
    \longrightarrow
    \Cinfty(C)[[\lambda]]
\end{equation}
becomes an isomorphism thanks to the above assumption that
$\deform{\iota^*}$ starts with $\iota^*$ in zeroth order of
$\lambda$. Now we can proceed as in the classical case by considering
the \emph{normalizer}, now in the associative sense, i.e.
\begin{equation}
    \label{eq:NormalizerQuantum}
    \qalgebra{B}_C
    =
    \normalizer(\qalgebra{J}_C)
    =
    \left\{
        f \in \Cinfty(M)[[\lambda]]
        \; \big\vert \;
        g \star f \in \qalgebra{I}_C
        \textrm{ for all }
        g \in \qalgebra{I}_C
    \right\}.
\end{equation}
Note that this is equivalent to the condition
$[f, g]_\star \in \qalgebra{I}_C$ for all $g \in \qalgebra{I}_C$ since
$\qalgebra{I}_C$ is already a left ideal. A simple check shows that
$\qalgebra{B}_C$ is a subalgebra of $\Cinfty(M)[[\lambda]]$ and
$\qalgebra{I}_C$ is a two-sided ideal in $\qalgebra{B}_C$. Thus it is
tempting to define the reduction on the quantum side as the quotient
algebra $\qalgebra{A}_\red = \qalgebra{B}_C \big/ \qalgebra{I}_C$ in
complete analogy to the above classical case.

We point out now an alternative but equivalent definition of this
reduced algebra:
\begin{proposition}
    \label{proposition:ReductionModuleAlaMartin}%
    The functions on the constraint surface become a left module of
    the algebra $\qalgebra{A}$ by \eqref{eq:FunctionsOnC} in a
    canonical way. Moreover, the module endomorphisms of this left
    module are isomorphic to the opposite algebra of
    $\qalgebra{A}_\red$ via
    \begin{equation}
        \label{eq:EndoAmodICtoBCmodIC}
        \faktor{\qalgebra{B}_C}{\qalgebra{I}_C}
        \ni [f]
        \; \longmapsto \;
        ([g] \longmapsto [g \star f])
        \in
        \End_{\qalgebra{A}}(\Cinfty(C)[[\lambda]])^\opp,
    \end{equation}
    where $[g]$ denotes an equivalence class in the quotient
    \eqref{eq:FunctionsOnC}.
\end{proposition}

This idea from \cite{bordemann:2005a} puts the role of the constraint
surface in a much clearer light: it carries a bimodule structure for
the original big algebra $\qalgebra{A}$ acting from the left and the
reduced algebra acting from the right such that the reduced algebra
coincides with (the opposite of) the module endomorphisms. Note,
however, that the $\qalgebra{A}_\red$-endomorphisms contain
$\qalgebra{A}$ but are typically strictly larger, see
\cite{gutt.waldmann:2010a}. In particular, this bimodule is typically
\emph{not} a Morita equivalence bimodule.

It is now a final check necessary to show that $\qalgebra{A}_\red$
defined as above actually gives a deformation quantization of the
reduced Poisson manifold $(M_\red, \pi_\red)$. This is by far not
trivial and in fact not true in general. Simple examples of
ill-adjusted star products on $M$ where this fails are discussed
e.g. in \cite{bordemann.herbig.waldmann:2000a}. More profound
obstructions are discussed in \cite{bordemann:2005a,
  bordemann:2004a:pre} in the symplectic case and in
\cite{cattaneo.felder:2007a, cattaneo.felder:2004a, cattaneo:2004a}
for the Poisson case.

However, in many reasonably nice situations the program can be carried
through and yields a star product $\star_\red$ for $M_\red$, see again
\cite{bordemann:2005a, bordemann:2004a:pre}: whenever the reduced
space $M_\red$ exists in the case of a symplectic manifold $M$, then
one can find a suitable star product $\star$ on $M$ for which the
construction yields a reduced star product $\star_\red$. A more
specific situation is the analog of the Marsden-Weinstein reduction,
the construction relying on BRST cohomological arguments
\cite{bordemann.herbig.waldmann:2000a}:
\begin{example}[Marsden-Weinstein reduction, quantum]
    \label{example:QuantumMarsdenWeinstein}%
    Suppose that we are in the same situation as in
    Example~\ref{example:MarsdenWeinstein}. Suppose moreover, that
    $\star$ is a star product on $M$, invariant under the group action
    $\Phi$ which allows for a quantum momentum map
    $\deform{J} = J + \sum_{r=1}^\infty \lambda^r J_r$, i.e. one has
    $[\deform{J}_\xi, f]_\star = \I\lambda \xi_M f$ for all
    $f \in \Cinfty(M)[[\lambda]]$ and
    $[\deform{J}_\xi, \deform{J}_\eta]_\star = \I\lambda
    \deform{J}_{[\xi, \eta]}$.
    Then one can construct a deformation $\deform{\iota^*}$ as needed
    and $\qalgebra{A}_\red$ turns out to be isomorphic to
    $\Cinfty(M_\red)[[\lambda]]$ as $\mathbb{C}[[\lambda]]$-module
    inducing thereby a star product $\star_\red$. Moreover, explicit
    formulas for the bimodule structure on $\Cinfty(C)[[\lambda]]$ can
    be given, see e.g. \cite{gutt.waldmann:2010a}. The existence of
    invariant star products, quantum momentum maps, and the
    corresponding reduction is discussed in detail in
    \cite{mueller-bahns.neumaier:2004a, gutt.rawnsley:2003a,
      fedosov:1998a} culminating in classification results
    \cite{reichert:2017a, reichert.waldmann:2016a} in the symplectic
    case.  Finally, the existence and classification of equivariant
    star products in the Poisson case has been recently proved in
    \cite{esposito.dekleijn.schnitzer:2018a:pre}.
\end{example}

We will come back to this construction at several instances. It will
serve as the main motivation in the following: based on these
observations we shall put the reduction process into a purely
algebraic framework.

%
%

\section{Coisotropic Triples and Pairs of Algebras}
\label{sec:CoisotropicTriplesPairsAlgebras}

In the following we fix a commutative unital ring $\mathbb{k}$ of
scalars which in most situations will be even a field. It will
sometimes be convenient to assume $\mathbb{Q} \subseteq \mathbb{k}$.
All occurring algebras and modules will be over $\mathbb{k}$ and
linearity always will include linearity over $\mathbb{k}$.

As we want to discuss reduction with respect to some coisotropic data,
we start with some unital \emph{total algebra} $\algebra{A}_\Total$ in
which we suppose to have a left ideal
$\algebra{A}_\Null \subseteq \algebra{A}_\Total$. The correspondence
with the above geometric situation is that the total algebra stands
for the functions on the total phase space while the left ideal
corresponds to the functions vanishing on the constraint surface. To
ensure maximal flexibility, we need to specify an additional algebra,
which we call the \emph{weakly observables} according to Dirac's
original discussion of constraint systems in \cite{dirac:1950a,
  dirac:1964a}. We thus have to specify a unital subalgebra
$\algebra{A}_\Wobs \subseteq \algebra{A}_\Total$ containing the left
ideal $\algebra{A}_\Null$ as a two-sided ideal
$\algebra{A}_\Null \subseteq \algebra{A}_\Wobs$. With other words,
$\algebra{A}_\Wobs$ will be a unital subalgebra of the normalizer of
the left ideal, i.e.
\begin{equation}
    \label{eq:NullInWobsInNormalizer}
    \algebra{A}_\Null
    \subseteq
    \algebra{A}_\Wobs
    \subseteq
    \normalizer(\algebra{A}_\Null).
\end{equation}
However, note that we explicitly allow $\algebra{A}_\Wobs$ to be
strictly smaller than the normalizer of $\algebra{A}_\Null$: in the
commutative situation above we took the Lie normalizer with respect to
the additional Poisson bracket, a structure we do not want to
introduce at this level. Summarizing, this leads now to the following
definition of a coisotropic triple of algebras:
\begin{definition}[Coisotropic triple of algebras]
    \label{definition:CoisotropicTripleAlgebras}%
    A coisotropic triple of algebras over $\mathbb{k}$ is a triple
    $\algebra{A} = (\algebra{A}_\Total, \algebra{A}_\Wobs,
    \algebra{A}_\Null)$
    consisting of a unital algebra $\algebra{A}_\Total$, a unital
    subalgebra $\algebra{A}_\Wobs \subseteq \algebra{A}_\Total$, and a
    left ideal $\algebra{A}_\Null \subseteq \algebra{A}_\Total$ such
    that $\algebra{A}_\Null \subseteq \algebra{A}_\Wobs$ is a
    two-sided ideal.
\end{definition}

With this definition, our geometric situation provides us some first
and guiding examples:
\begin{example}[Coisotropic triples of algebras from geometry]
    \label{example:CoisotropicTripleAlgbra}%
    Let $M$ be a Poisson manifold with a coisotropic submanifold
    $\iota\colon C \longrightarrow M$.
    \begin{examplelist}
    \item \label{item:CoisotropicClassicalAgain} Setting
        $\algebra{A}_\Total = \Cinfty(M)$ and
        $\algebra{A}_\Null = \algebra{J}_C = \ker\iota^*$ gives a
        total algebra and a (left) ideal inside. However, since
        $\algebra{A}_\Total$ is commutative, taking the normalizer of
        $\algebra{A}_\Null$ would reproduce $\algebra{A}_\Total$, a
        too simple choice to be interesting. However, taking
        $\algebra{A}_\Wobs = \algebra{B}_C$ gives an interesting
        coisotropic triple of algebras over $\mathbb{C}$.
    \item \label{item:QuantumTriple} Suppose now in addition that we
        can find a star product $\star$ and a deformation
        $\deform{\iota^*}$ of the restriction map. Then
        $\algebra{A}_\Total = \Cinfty(M)[[\lambda]]$, equipped with
        the star product $\star$ serves as total algebra,
        $\algebra{A}_\Null = \ker\deform{\iota^*}$ will be the left
        ideal and the normalizer
        $\algebra{A}_\Wobs = \normalizer(\algebra{A}_\Null)$ can be
        taken as weakly observables. We obtain a truly noncommutative
        coisotropic triple is this case, over the ring
        $\mathbb{C}[[\lambda]]$ as underlying scalars.
    \end{examplelist}
\end{example}
\begin{example}
    \label{example:Ciccoli}%
    We recall a concrete example of coisotropic triples already
    discussed in \cite{ciccoli:1997a} following
    \cite{ciccoli:2019a:privat}. Let $\mathcal{E}_q (2)$ be the
    $*$-algebra generated by the following relations:
    \begin{align*}
        v v^{-1} &= 1 = v^{-1}v
        \\
        v n^* &= q^{-1}nv
        \\
        vn &= q nv
        \\
        n n^* &= n^* n,
    \end{align*}
    where $q$ is a real parameter. Denote by $\mathcal I_\lambda$
    the right ideal generated by $\lambda (v -1) + n$ and
    $\bar{\lambda} (v^{-1}-1) + n^*$, with $\lambda \in \mathbb C$.
    It is easy to see that $\mathcal{E}_q (2)$, $\mathcal I_\lambda$
    and the corresponding normalizer form a coisotropic triple.
\end{example}

Before investigating more examples and constructions we introduce the
following notion of morphisms between coisotropic triples of algebras:
we define a \emph{morphism} from the coisotropic triple
$\algebra{A} = (\algebra{A}_\Total, \algebra{A}_\Wobs,
\algebra{A}_\Null)$
to the coisotropic triple
$\algebra{B} = (\algebra{B}_\Total, \algebra{B}_\Wobs,
\algebra{B}_\Null)$ to be a unital algebra morphism
\begin{equation}
    \label{eq:MorphismTripleAlgebra}
    \Phi\colon
    \algebra{A}_\Total \longrightarrow \algebra{B}_\Total,
\end{equation}
such that
\begin{equation}
    \label{eq:MorphismWobsNullCompatible}
    \Phi(\algebra{A}_\Wobs) \subseteq \algebra{B}_\Wobs
    \quad
    \textrm{and}
    \quad
    \Phi(\algebra{A}_\Null) \subseteq \algebra{B}_\Null.
\end{equation}
It is clear that the composition of morphisms is again a morphism and
hence we ultimately obtain a category of coisotropic triples of
algebras which we denote by $\CoisoAlgTriple$ or
$\CoisoAlgTriple_{\mathbb{k}}$ if we want to emphasize the underlying
ring $\mathbb{k}$ of scalars.
\begin{remark}
    \label{remark:PairPointOfView}%
    If one would only focus on the pair
    $(\algebra{A}_\Total, \algebra{A}_\Null)$ then this notion of
    morphisms becomes less obvious: in that case a natural candidate
    for a morphism from one such pair to another would be a unital
    algebra morphism
    $\Phi\colon \algebra{A}_\Total \longrightarrow \algebra{B}_\Total$
    with $\Phi(\algebra{A}_\Null) \subseteq \algebra{B}_\Null$.
    However, simple examples show that then the normalizer
    $\normalizer(\algebra{A}_\Null)$ needs not to be mapped into the
    normalizer $\normalizer(\algebra{B}_\Null)$.  As we will base many
    constructions on the choice of $\algebra{A}_\Wobs$, we need to
    take care of this part of the triple by hand.
\end{remark}

We denote the category of unital algebras (with unital algebra
morphisms) by $\Algebras$ while the not necessarily unital algebras
are then denoted by $\algebras$. Then we have several obvious
functors. First, projecting on one of the three components of a triple
is of course functorial leading to functors
\begin{gather}
    \label{eq:FunctorsTripleToAlg}
    \CoisoAlgTriple \ni \algebra{A}
    \; \longmapsto \;
    \algebra{A}_\Total \in \Algebras
    \quad
    \textrm{and}
    \quad
    \CoisoAlgTriple \ni \algebra{A}
    \; \longmapsto \;
    \algebra{A}_\Wobs \in \Algebras
    \\
    \shortintertext{as well as}
    \label{eq:NullFunctor}
    \CoisoAlgTriple \ni \algebra{A}
    \; \longmapsto \;
    \algebra{A}_\Null \in \algebras,
\end{gather}
each with the obvious restriction of morphisms. But we can also go the
other way and build coisotropic triples out of single algebras. Here
we have several options. The first is the \emph{trivial triple}
\begin{equation}
    \label{eq:TrivialTriple}
    \algebra{A}_\Trivial = (\algebra{A}, \algebra{A}, \algebra{A})
\end{equation}
for a unital algebra $\algebra{A}$. Alternatively, we can construct
the \emph{un-reduce triple}
\begin{equation}
    \label{eq:UnReduceTriple}
    \algebra{A}_\Unreduce = (\algebra{A}, \algebra{A}, \{0\})
\end{equation}
for a unital algebra $\algebra{A}$. Both versions yield functors
$\Algebras \longrightarrow \CoisoAlgTriple$. Finally, more important
for our original motivation, is the \emph{Dirac triple} we can build
out of a pair of a unital algebra $\algebra{A}$ and a left ideal
$\algebra{J} \subseteq \algebra{A}$. Here we set
\begin{equation}
    \label{eq:CanonicalTriple}
    \algebra{A}_{\scriptstyle{\mathrm{Dirac}}}
    =
    (\algebra{A}, \normalizer(\algebra{J}), \algebra{J}).
\end{equation}
In view of Remark~\ref{remark:PairPointOfView} this becomes again
functorial if we consider the category $\LeftIdeals$ of pairs of
unital algebras with left ideals together with unital algebra
morphisms mapping one left ideal into the other \emph{and} mapping the
normalizer of the first left ideal into the normalizer of the
second. Then the canonical triple becomes a functor
\begin{equation}
    \label{eq:NormalizeFunctor}
    \functor{Dirac}\colon \LeftIdeals \longrightarrow \CoisoAlgTriple.
\end{equation}

After having established the notion of coisotropic triples one can
define the reduction of them in the following way, mimicking the
situation of star products in the geometric situation: Let
$\algebra{A} = (\algebra{A}_\Total, \algebra{A}_\Wobs,
\algebra{A}_\Null)$
be a coisotropic triple of algebras. Then the reduction of
$\algebra{A}$ is defined to be the unital algebra
\begin{equation}
    \label{eq:ReductionAlgebra}
    \algebra{A}_\red
    =
    \faktor{\algebra{A}_\Wobs}{\algebra{A}_\Null}.
\end{equation}
For a morphism $\Phi\colon \algebra{A} \longrightarrow \algebra{B}$
between coisotropic triples of algebras we see that the restriction of
$\Phi$ to the weak observables passes to the quotient and thus defines
a unital algebra morphism
\begin{equation}
    \label{eq:PhiReduced}
    \Phi_\red\colon
    \algebra{A}_\red \longrightarrow \algebra{B}_\red.
\end{equation}
Clearly, this yields a functorial reduction:
\begin{proposition}
    \label{proposition:CoisotropicTripleAlgebraReduction}%
    Reduction of coisotropic triples of algebras yields a functor
    \begin{equation}
        \label{eq:ReductionFunctorAlgebras}
        \reduce\colon
        \CoisoAlgTriple \longrightarrow \Algebras.
    \end{equation}
\end{proposition}
\begin{corollary}
    \label{corollary:UnreduceTrivialTriple}%
    The reduction of the un-reduce triple is naturally isomorphic to
    the identity functor on $\Algebras$. The reduction of the trivial
    triple is naturally isomorphic to the zero-functor on $\Algebras$
    sending an algebra to the zero algebra $\{0\}$.
\end{corollary}

Surprisingly, the reduction uses only the information of the pair
$(\algebra{A}_\Wobs, \algebra{A}_\Null)$ instead of the full
triple. The ambient total algebra does not play a role here. This
raises of course the question whether one can not just start with a
category of \emph{coisotropic pairs} consisting of a unital algebra
together with a two-sided ideal inside. To some extend this is true
and many of the following constructions will only use the pair instead
of the triple. Thus we also state the definition of a coisotropic pair
as follows:
\begin{definition}[Coisotropic pair]
    \label{definition:CoisotropicPair}%
    A coisotropic pair of algebras is a pair
    $\algebra{A} = (\algebra{A}_\Wobs, \algebra{A}_\Null)$ of a unital
    associative algebra $\algebra{A}_\Wobs$ over $\mathbb{k}$ together
    with a two-sided ideal
    $\algebra{A}_\Null \subseteq \algebra{A}_\Wobs$. A morphism
    between two coisotropic pairs $\algebra{A}$ and $\algebra{B}$ is a
    unital algebra morphism
    $\Phi\colon \algebra{A}_\Wobs \longrightarrow \algebra{B}_\Wobs$
    with $\Phi(\algebra{A}_\Null) \subseteq \algebra{B}_\Null$.
\end{definition}
Clearly, this gives again a categorical framework for coisotropic
pairs of algebras. We denote the resulting category by $\CoisoAlgPair$
or $\CoisoAlgPair_{\mathbb{k}}$ whenever we need to emphasize the
underlying ring of scalars.

Forgetting the total algebra yields then a functor
\begin{equation}
    \label{eq:CoisoTripleToPairs}
    \CoisoAlgTriple \longrightarrow \CoisoAlgPair.
\end{equation}
This also results in a functor
\begin{equation}
    \label{eq:LeftIdealsCoisoPairs}
    \functor{Dirac}\colon \LeftIdeals \longrightarrow \CoisoAlgPair,
\end{equation}
sending an algebra $\algebra{A}_\Total$ with left ideal
$\algebra{A}_\Null \subseteq \algebra{A}_\Total$ to the coisotropic
pair $(\normalizer(\algebra{A}_\Null), \algebra{A}_\Null)$. Note that
here the correct definition of morphisms in $\LeftIdeals$ is crucial
to make this functorial. Conversely, we can extend a coisotropic pair
$(\algebra{A}_\Wobs, \algebra{A}_\Null)$ always to a coisotropic
triple in a trivial way by mapping it to
$(\algebra{A}_\Wobs, \algebra{A}_\Wobs, \algebra{A}_\Null)$.
Forgetting the triple gives back the pair we started with.  In fact
this is a left adjoint to the forgetful functor forgetting
$\algebra{A}_\Total$.  Moreover, this allows us to view
$\CoisoAlgPair$ as a subcategory of $\CoisoAlgTriple$.  We will often
give definitions only for coisotropic triples, and the appropriate
definitions for coisotropic pairs will then be given by restricting to
this subcategory.  However, the interesting triples are those where
$\algebra{A}_\Null \subseteq \algebra{A}_\Total$ is \emph{only} a left
ideal. Hence these will not show up as images of this inclusion
functor $\CoisoAlgPair \longrightarrow \CoisoAlgTriple$.

Note also that we have a \emph{trivial} coisotropic pair for every
unital algebra $\algebra{A}$ by setting
\begin{equation}
    \label{eq:TrivialPair}
    \algebra{A}_\Trivial = (\algebra{A}, \algebra{A})
\end{equation}
as well as the \emph{un-reduce pair}
\begin{equation}
    \label{eq:UnreducePair}
    \algebra{A}_{\scriptstyle{\mathrm{unred}}} = (\algebra{A}, \{0\}).
\end{equation}
Both notions are of course compatible with the trivial and the
un-reduce triples and the functor \eqref{eq:CoisoTripleToPairs}.

As mentioned before, the reduction functor only uses the information
of a pair and thus gives a reduction
$\reduce\colon \CoisoAlgPair \longrightarrow \Algebras$. Ultimately,
we arrive at the following diagram
\begin{equation}
    \label{eq:ReductionFunctorsCommute}
    \begin{tikzpicture}[baseline=(current bounding box.center)]
        \matrix(m) [matrix of math nodes,
        row sep=4em,
        column sep=6em]{
          & \CoisoAlgTriple & \\
          \LeftIdeals & & \Algebras \\
          & \CoisoAlgPair \\
        };
        \draw[->]
        (m-2-1) to node[above, sloped] {$\functor{Dirac}$} (m-1-2);
        \draw[->]
        (m-2-1) to node[below, sloped] {$\functor{Dirac}$} (m-3-2);
        \draw[->, bend left = 30pt]
        (m-2-3) to node[above, sloped] {$\functor{unred}$} (m-1-2);
        \draw[->, bend right = 30pt]
        (m-2-3) to node[above,sloped] {$\functor{trivial}$} (m-1-2);
        \draw[->]
        (m-1-2) to node[above,sloped] {$\reduce$} (m-2-3);
        \draw[->, bend left = 30pt]
        (m-2-3) to node[below, sloped] {$\functor{unred}$} (m-3-2);
        \draw[->, bend right = 30pt]
        (m-2-3) to node[below,sloped] {$\functor{trivial}$} (m-3-2);
        \draw[->]
        (m-3-2) to node[below,sloped] {$\reduce$} (m-2-3);
        \draw[->, bend left = 20pt]
        (m-1-2) to node[above, sloped] {$\functor{forget}$} (m-3-2);
        \draw[->, bend left = 20pt]
        (m-3-2) to node[above, sloped] {$\functor{trivial}$} (m-1-2);
    \end{tikzpicture}
\end{equation}
of functors. Thus, in particular, the reduction of the un-reduce pair
reproduces the algebra one started with and the reduction of the
trivial triple is the zero algebra.

\begin{remark}
    \label{remark:FromLeftIdealToTriplesandPairs}%
    While the pair point of view simplifies the reduction picture
    drastically, the original motivation is to generalize the
    geometric situation of phase space reduction: there the algebra in
    the middle $\algebra{A}_\Wobs$ is typically the most difficult one
    to get, both in the classical and the quantum situation. Instead,
    one starts with the ambient algebra $\algebra{A}_\Total$. Then in
    the quantum version it is already difficult enough (and sometimes
    obstructed) to deform the classical vanishing ideal into a left
    ideal $\algebra{A}_\Null$. Only in the last step one can then find
    $\algebra{A}_\Wobs$. Thus we will discuss triples and pairs in
    parallel to keep in mind that a serious application will always
    require to actually find the triples out of more simple
    data. Ultimately, we will be interested in starting with
    $(\algebra{A}_\Total, \algebra{A}_\Null)$ in $\LeftIdeals$ and
    construct the relevant data out of this. While for the algebras
    there is a functorial way by using the normalizers, in the case of
    (bi-)modules we will see that one typically has no obvious
    functorial way to accomplish this.
\end{remark}

Before moving to the categories of (bi-)modules over coisotropic
triples and pairs, we mention the following canonical bimodule
relating the total algebra and the reduced one:
\begin{proposition}
    \label{proposition:CanonicalBimodule}%
    Let
    $\algebra{A} = (\algebra{A}_\Total, \algebra{A}_\Wobs,
    \algebra{A}_\Null)$
    be a coisotropic triple of algebras over $\mathbb{k}$.
    \begin{propositionlist}
    \item \label{item:EndCAisAredOpp} Then
        \begin{equation}
            \label{eq:CanonicalBimoduleAtotalAred}
            \functor{C}(\algebra{A})
            =
            \faktor{\algebra{A}_\Total}{\algebra{A}_\Null}
        \end{equation}
        is a $(\algebra{A}_\Total, \algebra{A}_\red)$-bimodule, cyclic
        with respect to $\algebra{A}_\Total$, and one has
        \begin{equation}
            \label{eq:EndosOfCanonicalBimodule}
            \End_{\algebra{A}_\Total}(\functor{C}(\algebra{A}))^\opp
            =
            \faktor{\normalizer(\algebra{A}_\Null)}{\algebra{A}_\Null}.
        \end{equation}
    \item For a morphism $\Phi\colon \algebra{A} \longrightarrow
        \algebra{B}$ of coisotropic triples of algebras the map
        \begin{equation}
            \label{eq:CPhi}
            \functor{C}(\Phi)\colon
            \functor{C}(\algebra{A}) \ni [a]
            \; \longmapsto \;
            [\Phi(a)] \in \functor{C}(\algebra{B})
        \end{equation}
        is a bimodule morphism along the two algebra morphisms
        $\Phi_\Total\colon \algebra{A}_\Total \longrightarrow
        \algebra{B}_\Total$
        and
        $\Phi_\red\colon \algebra{A}_\red \longrightarrow
        \algebra{B}_\red$.
    \item \label{item:CAisFunctorial} Mapping $\algebra{A}$ to
        $\functor{C}(\algebra{A})$ and morphisms
        $\Phi\colon \algebra{A} \longrightarrow \algebra{B}$ to
        $\functor{C}(\Phi)$ gives a functor
        \begin{equation}
            \label{eq:Cfunctor}
            \functor{C}\colon
            \CoisoAlgTriple
            \longrightarrow
            \Bimodule
        \end{equation}
        into the category $\Bimodule$ of bimodules with morphisms
        being bimodule morphisms along algebra morphisms of the
        involved algebras.
    \end{propositionlist}
\end{proposition}
\begin{proof}
    Indeed, since $\algebra{A}_\Null$ is a left ideal in
    $\algebra{A}_\Total$, the quotient $\functor{C}(\algebra{A})$ is a
    left $\algebra{A}$-module. Since $\algebra{A}_\Total$ is unital,
    $\functor{C}(\algebra{A})$ is cyclic with cyclic element
    $[\Unit] \in \module{C}(\algebra{A})$. As already indicated in the
    case of star products in \eqref{eq:EndoAmodICtoBCmodIC}, the
    opposite of the module endomorphisms of this left
    $\algebra{A}$-module is given by
    $\normalizer(\algebra{A}_\Null) \big/ \algebra{A}_\Null$ using the
    right multiplications. Since
    $\algebra{A}_\Wobs \subseteq \normalizer(\algebra{A}_\Null)$ by
    assumption, this gives the canonical right module structure,
    showing the first claim. For the second, we note that
    $[\Phi(a)] \in \functor{C}(\algebra{B})$ only depends on the class
    $[a]$ since $\Phi(\algebra{A}_\Null) \subseteq \algebra{B}_\Null$.
    Then it is clear that
    $\functor{C}(\Phi) (a \cdot [x]) = \Phi_\Total(a) \cdot
    \functor{C}(\Phi)([x])$
    and
    $\functor{C}(\Phi) ([x] \cdot [a']) = \functor{C}(\Phi)([x]) \cdot
    \Phi_\red([a'])$
    for $a, x \in \algebra{A}_\Total$ and $a' \in \algebra{A}_\Wobs$
    since we can check these relations on representatives. From this
    the second part follows. But then the claimed functoriality is
    clear.
\end{proof}

Geometrically, $\functor{C}(\algebra{A})$ corresponds to the functions
on the constraint surface. Even though in the classical (commutative)
case this is an algebra itself, we will consider it only as a
$(\algebra{A}_\Total, \algebra{A}_\red)$-bimodule, since this is the
only structure remaining in the noncommutative situation. Note that
here we need the coisotropic triples instead of mere coisotropic pairs
of algebras in order to define the bimodule
$\functor{C}(\algebra{A})$.
\begin{remark}
    \label{remark:ReductionMartin}%
    We should remark that this observation stands at the beginning of
    the reduction idea of Bordemann in \cite{bordemann:2005a,
      bordemann:2004a:pre} where the geometric situation is analyzed
    in detail, including a description of obstructions for the
    deformation quantization of coisotropic submanifolds in the
    symplectic framework.
\end{remark}

%
%

\section{Triples and Pairs of Bimodules}
\label{sec:TriplePairBimodules}

Let us come back to the geometric picture of
Section~\ref{sec:RedcutionPoissonGeometryDQ}, where
$\iota \colon C \longrightarrow M$ is a coisotropic submanifold of a
Poisson manifold $(M,\pi)$ and we assume to have a surjective
submersion $\pr \colon C \longrightarrow M_\red$.  Here $M_\red$ is
again the leaf space $C\big/\!\sim$ of the characteristic distribution
of $C$.  Now let in addition $p \colon E \longrightarrow M$ be a
vector bundle over $M$.  Then we know that
$\module{E}_\Total = \Secinfty(E)$ is a $\Cinfty(M)$-module and we can
define a submodule
\begin{equation}
    \label{eq:GeometricENull}
    \module{E}_\Null
    =
    \left\{
        s \in \Secinfty(E)
        \; \big\vert \;
        s \at{C} = 0
    \right\}
\end{equation}
of all sections vanishing on $C$.  In order to define a reduced vector
bundle $p_\red \colon E_\red \longrightarrow M_\red$ we would like to
use the sections of $E$ that are constant along the characteristic
distribution of $C$.  Of course, there is no canonical way to make
sense out of such a statement. Instead, we need to use some additional
data.  For this, let $\nabla$ be a covariant derivative for the vector
bundle $E$ and consider those sections of $E$ whose covariant
derivative in the direction of Hamiltonian vector fields of functions
in the vanishing ideal $\algebra{I}_C$ vanish on $C$. We denote this
subset by
\begin{equation}
    \label{eq:GeometricEWobs}
    \module{E}_\Wobs
    =
    \left\{
        s \in \Secinfty(E)
        \; \big\vert \;
        (\nabla_{X_f}s)\at{C} = 0
        \textrm{ for all }
        f \in \algebra{I}_C
    \right\}.
\end{equation}
Note that for the definition of $\module{E}_\Wobs$ we used the
additional information of a covariant derivative on $E$ while
$\module{E}_\Null$ was still canonical.  This is different from
coisotropic algebras, where we could define $\algebra{A}_\Wobs$ as the
Poisson normalizer $\algebra{B}_C$.  We use this geometric situation
as motivation for the definition of bimodules over coisotropic
triples.
\begin{definition}[Bimodules over coisotropic triples]
    \label{definition:CoisoTriplesBimodules}%
    Let $\algebra{A}$ and $\algebra{B}$ be coisotropic triples of
    algebras over $\mathbb{k}$.
    \begin{definitionlist}
    \item \label{item:CoisoTripleBimodule} A triple
        $\module{E} = (\module{E}_\Total, \module{E}_\Wobs,
        \module{E}_\Null)$
        consisting of a
        $(\algebra{B}_\Total, \algebra{A}_\Total)$-bimodule
        $\module{E}_\Total$ and
        $(\algebra{B}_\Wobs, \algebra{A}_\Wobs)$-bimodules
        $\module{E}_\Wobs$ and $\module{E}_\Null$ together with a
        bimodule morphism
        $\iota_\module{E} \colon \module{E}_\Wobs \longrightarrow
        \module{E}_\Total$
        along the inclusions
        $\algebra{B}_\Wobs \subseteq \algebra{B}_\Total$ and
        $\algebra{A}_\Wobs \subseteq \algebra{A}_\Total$ is called a
        $(\algebra{B},\algebra{A})$-bimodule if
        $\module{E}_\Null \subseteq \module{E}_\Wobs$ is a
        sub-bimodule such that
        \begin{equation}
            \label{eq:BimoduleTripleCondition}
            \algebra{B}_\Null \cdot \module{E}_\Wobs
            \subseteq
            \module{E}_\Null
            \quad
            \textrm{and}
            \quad
            \module{E}_\Wobs \cdot \algebra{A}_\Null
            \subseteq
            \module{E}_\Null.
        \end{equation}
    \item \label{item:CoisoTripleBimoduleMorphism} A morphism
        $\Phi\colon \module{E} \longrightarrow \tilde{\module{E}}$
        between $(\algebra{B}, \algebra{A})$-bimodules is a pair
        $(\Phi_\Total, \Phi_\Wobs)$ of a
        $(\algebra{B}_\Total, \algebra{A}_\Total)$-bimodule morphism
        $\Phi_\Total \colon \module{E}_\Total \longrightarrow
        \tilde{\module{E}}_\Total$
        and $(\algebra{B}_\Wobs,\algebra{A}_\Wobs)$-bimodule morphism
        $\Phi\colon \module{E}_\Wobs \longrightarrow
        \tilde{\module{E}}_\Wobs$
        such that
        $\Phi_\Total \circ \iota_\module{E} =
        \iota_{\tilde{\module{E}}} \circ \Phi_\Wobs$
        and
        $\Phi_\Wobs(\module{E}_\Null) \subseteq
        \tilde{\module{E}}_\Null$.
    \item \label{item:CategoryCoisoTripleBimods} The category of
        $(\algebra{B}, \algebra{A})$-bimodules is denoted by
        $\CoisoBimodTriple(\algebra{B}, \algebra{A})$.
    \end{definitionlist}
\end{definition}

The motivation is to mimic the two-sided ideal property of the
$\NULL$-component also on the level of (bi-)modules. It is of course
a routine check that the obvious composition of morphisms is again a
morphism and we thus get a category.  The reason we did not choose to
require an inclusion $\module{E}_\Wobs \subseteq \module{E}_\Total$ is
that with this broader notion tensor products turn out to become
easier as we will not have to insist on flatness in order to guarantee
the injectivity of the tensor product of the inclusions.
Nevertheless, viewing a coisotropic triple $\algebra{A}$ as bimodule
over itself
$\iota_\algebra{A} \colon \algebra{A}_\Wobs \longrightarrow
\algebra{A}_\Total$
is still the inclusion map.  Similar to the case of coisotropic
triples of algebras we can also define modules over coisotropic pairs
by simply ignoring the $\TOTAL$-component.
\begin{definition}[Bimodules over coisotropic pairs]
    \label{definition:CoisoPairsBimodules}%
    Let $\algebra{A}$ and $\algebra{B}$ be coisotropic pairs of
    algebras over $\mathbb{k}$.
    \begin{definitionlist}
    \item \label{item:CoisoPairBimodule} A pair
        $\module{E} = (\module{E}_\Wobs, \module{E}_\Null)$ of
        $(\algebra{B}_\Wobs, \algebra{A}_\Wobs)$-bimodules is called a
        $(\algebra{B},\algebra{A})$-bimodule if
        $\module{E}_\Null \subseteq \module{E}_\Wobs$ is a
        sub-bimodule such that
        \begin{equation}
            \label{eq:BimodulePairCondition}
            \algebra{B}_\Null \cdot \module{E}_\Wobs
            \subseteq
            \module{E}_\Null
            \quad
            \textrm{and}
            \quad
            \module{E}_\Wobs \cdot \algebra{A}_\Null
            \subseteq
            \module{E}_\Null.
        \end{equation}
    \item \label{item:CoisoPairBimoduleMorphism} A morphism
        $\Phi\colon \module{E} \longrightarrow \tilde{\module{E}}$
        between $(\algebra{B}, \algebra{A})$-bimodules is a
        $(\algebra{B}_\Wobs, \algebra{A}_\Wobs)$-bimodule morphism
        $\Phi\colon \module{E}_\Wobs \longrightarrow
        \tilde{\module{E}}_\Wobs$
        such that
        $\Phi(\module{E}_\Null) \subseteq \tilde{\module{E}}_\Null$.
    \item \label{item:CategoryCoisoPairBimods} The category of
        $(\algebra{B}, \algebra{A})$-bimodules is denoted by
        $\CoisoBimodPair(\algebra{B}, \algebra{A})$.
    \end{definitionlist}
\end{definition}

By forgetting the total bimodule we get a forgetful functor
\begin{equation}
	\CoisoBimodTriple \longrightarrow \CoisoBimodPair.
\end{equation}
Conversely, we can go the other way by mapping a
$(\algebra{B},\algebra{A})$-bimodule $\module{E}$ over coisotropic
pairs $\algebra{A}$ and $\algebra{B}$ to the bimodule
$(\module{E}_\Wobs, \module{E}_\Wobs, \module{E}_\Null)$ over the
coisotropic triples
$(\algebra{A}_\Wobs, \algebra{A}_\Wobs, \algebra{A}_\Null)$ and
$(\algebra{B}_\Wobs, \algebra{B}_\Wobs, \algebra{B}_\Null)$.  Similar
to the case of coisotropic algebras this is a left adjoint to the
forgetful functor forgetting the $\TOTAL$-component.

The bicategory $\Bimodules$ of bimodules over algebras with the tensor
product as composition functors is one of the most basic examples of a
bicategory.  The goal of this section is to prove that we can
construct bicategories $\CoisoBimodTriple$ and $\CoisoBimodPair$
building on the above categories as well. Thus we can realize
$\CoisoBimodPair$ as a sub-bicategory of $\CoisoBimodTriple$. To show
this we need to define a tensor product of bimodules over coisotropic
triples and pairs and to check that there exist natural
transformations of associativity as well as left and right identities.
This is done in the following lemmas.  As a first step we construct a
tensor product functor
\begin{equation}
    \label{eq:TensorProductOfTriples}
    \tensor[\algebra{B}]\colon
    \CoisoBimodTriple{(\algebra{C}, \algebra{B})}
    \times
    \CoisoBimodTriple{(\algebra{B}, \algebra{A})}
    \longrightarrow
    \CoisoBimodTriple{(\algebra{C}, \algebra{A})}
\end{equation}
by tensoring the components of the triple as follows:
\begin{lemma}
    \label{lemma:TensorProductBimodules}%
    Let $\algebra{A}$, $\algebra{B}$ and $\algebra{C}$ be coisotropic
    triples of algebras and let
    $\module{F} \in \CoisoBimodTriple{(\algebra{C}, \algebra{B})}$,
    $\module{E} \in \CoisoBimodTriple{(\algebra{B},\algebra{A})}$ be
    corresponding bimodules. Then
    $\algbimodule{C}{F}{B} \tensor[\algebra{B}] \algbimodule{B}{E}{A}$
    given by the components
    \begin{align}
        \label{eq:TotalComponentTensorProduct}
        \left(
            \algbimodule{C}{F}{B}
            \tensor[\algebra{B}]
            \algbimodule{B}{E}{A}
        \right)_{\Total}
        &=
        \module{F}_{\Total}
        \tensor[\algebra{B}_{\Total}]
        \module{E}_{\Total}, \\
        \label{eq:WobsComponentTensorProduct}
        \left(
            \algbimodule{C}{F}{B}
            \tensor[\algebra{B}]
            \algbimodule{B}{E}{A}
        \right)_{\Wobs}
        &=
        \module{F}_{\Wobs}
	\tensor[\algebra{B}_{\Wobs}]
        \module{E}_{\Wobs}, \\
        \label{eq:NullComponentTensorProduct}
        \left(
            \algbimodule{C}{F}{B}
            \tensor[\algebra{B}]
            \algbimodule{B}{E}{A}
        \right)_{\Null}
        &=
        \module{F}_{\Wobs}
        \tensor[\algebra{B}_\Wobs]
        \module{E}_{\Null}
        +
        \module{F}_{\Null}
        \tensor[\algebra{B}_\Wobs]
        \module{E}_\Wobs
    \end{align}
    is a $(\algebra{C}, \algebra{A})$-bimodule, where we use the
    tensor product
    $\iota_{\module{F} \tensor \module{E}} = \iota_{\module{F}}
    \tensor \iota_{\module{E}}$
    to map the $\WOBS$-component into the $\TOTAL$-component.
\end{lemma}
\begin{proof}
    Note that the tensor product
    $\module{F}_{\Wobs}\tensor[\algebra{B}_{\Wobs}]
    \module{E}_{\Null}$
    is \emph{not} a submodule of
    $\module{F}_{\Wobs} \tensor[\algebra{B}_{\Wobs}]
    \module{E}_{\Wobs}$
    directly, thus \eqref{eq:NullComponentTensorProduct} has to be
    understood in one of the two following equivalent ways: either
    view
    $\module{F}_{\Wobs} \tensor[\algebra{B}_{\Wobs}]
    \module{E}_{\Null}$
    as the submodule of
    $\module{F}_{\Wobs}\tensor[\algebra{B}_{\Wobs}]
    \module{E}_{\Wobs}$
    generated by all elements of the form $y\tensor x$, with
    $y \in \module{F}_{\Wobs}$, $x \in \module{E}_{\Null}$, or as the
    image of
    $\id_{\module{F}_{\Wobs}} \tensor \iota_{\module{E}_{\Null}}$
    where $\iota_{\module{E}_{\Null}}$ is the inclusion of
    $\module{E}_{\Null}$ into $\module{E}_{\Wobs}$. Similarly for
    $\module{F}_{\Null} \tensor[\algebra{B}_{\Wobs}]
    \module{E}_{\Wobs}$.
    Now observe that
    $(\module{F} \tensor[\algebra{B}] \module{E})_\Total$ is a
    $(\algebra{C}_\Total, \algebra{A}_\Total)$-bimodule and
    $(\module{F} \tensor[\algebra{B}] \module{E})_{\Wobs}$ and
    $(\module{F} \tensor[\algebra{B}] \module{E})_{\Null}$ are clearly
    $(\algebra{C}_{\Wobs}, \algebra{A}_{\Wobs})$-bimodules.  Moreover,
    the map
    $\iota_\module{F} \tensor \iota_\module{E} \colon (\module{F}
    \tensor[\algebra{B}] \module{E})_\Wobs \longrightarrow (\module{F}
    \tensor[\algebra{B}] \module{E})_\Total$
    is a bimodule morphism and
    \begin{equation*}
        \algebra{B}_{\Null}
        \cdot
        (\module{F} \tensor[\algebra{B}] \module{E})_{\Wobs}
        =
        (\algebra{B}_{\Null} \cdot \module{F}_{\Wobs})
        \tensor[\algebra{B}_{\Wobs}]
        \module{E}_{\Wobs}
        \subseteq
        \module{F}_{\Null}
        \tensor[\algebra{B}]_{\Wobs}
        \module{E}_{\Wobs}
        \subseteq
        (\module{F} \tensor[\algebra{B}] \module{E})_{\Null}
    \end{equation*}
    and
    \begin{equation*}
        (\module{F} \tensor[\algebra{B}] \module{E})_{\Wobs}
        \cdot
        \algebra{A}_{\Null}
        =
        \module{F}_{\Wobs}
        \tensor[\algebra{B}_{\Wobs}]
        (\module{E}_{\Wobs} \cdot \algebra{A}_{\Null})
        \subseteq
        \module{F}_{\Wobs}
        \tensor[\algebra{B}_{\Wobs}]
        \module{E}_{\Null}
        \subseteq
        (\module{F} \tensor[\algebra{B}] \module{E})_{\Null}
    \end{equation*}
    hold. Hence $\module{F} \tensor[\algebra{B}] \module{E}$ is a
    $(\algebra{B}, \algebra{A})$-bimodule.
\end{proof}
\begin{lemma}
    \label{lemma:TensorProductBimoduleMorphisms}%
    Let $\algebra{A}$, $\algebra{B}$ and $\algebra{C}$ be coisotropic
    triples of algebras. Moreover, let
    $\Psi\colon \module{F} \longrightarrow \module{F}^\prime$ and
    $\Phi \colon \module{E} \longrightarrow \module{E}^\prime$ be
    morphisms of $(\algebra{C}, \algebra{B})$-bimodules $\module{F}$,
    $\module{F}^\prime$, and $(\algebra{B}, \algebra{A})$-bimodules
    $\module{E}$, $\module{E}^\prime$, respectively.  Then
    $\Psi \tensor \Phi$ given by
    \begin{gather}
        \label{eq:MorphismTensorTotal}
        (\Psi \tensor \Phi)_\Total
        =
        \Psi_\Total \tensor \Phi_\Total \\
        \shortintertext{and}
        \label{eq:MorphismTensorWobs}
        (\Psi \tensor \Phi)_\Wobs
        =
        \Psi_\Wobs \tensor \Phi_\Wobs
    \end{gather}
    is a bimodule morphism from
    $\module{F} \tensor[\algebra{B}] \module{E}$ to
    $\module{F}^\prime \tensor[\algebra{B}] \module{E}^\prime$.
\end{lemma}
\begin{proof}
    It is clear that $(\Psi \tensor \Phi)_\Total$ and
    $(\Psi \tensor \Phi)_\Wobs$ are a morphism of
    $(\algebra{C}_\Total, \algebra{A}_\Total)$- and
    $(\algebra{C}_{\Wobs}, \algebra{A}_{\Wobs})$-bimodules,
    respectively, fulfilling
    $(\Psi \tensor \Phi)_\Total \circ (\iota_\module{F} \tensor
    \iota_\module{E}) = (\iota_{\module{F}'} \tensor
    \iota_{\module{E}'}) \circ (\Psi_\Wobs \tensor \Phi_\Wobs)$.
    Moreover.
    \begin{align*}
        (\Psi \tensor \Phi)_\Wobs
        \left(
            \module{F} \tensor[\algebra{B}] \module{E}
        \right)_{\Null}
        &=
        (\Psi \tensor \Phi)_\Wobs
        \left(
            \module{F}_{\Wobs}
            \tensor[\algebra{B}_{\Wobs}]
            \module{E}_{\Null}
            +
            \module{F}_{\Null}
            \tensor[\algebra{B}_{\Wobs}]
            \module{E}_{\Wobs}
        \right) \\
        &=
        \Psi_\Wobs(\module{F}_{\Wobs})
        \tensor[\algebra{B}_{\Wobs}]
        \Phi_\Wobs(\module{E}_{\Null})
        +
        \Psi_\Wobs(\module{F}_{\Null})
        \tensor[\algebra{B}_{\Wobs}]
        \Phi_\Wobs(\module{E}_{\Wobs}) \\
        &\subseteq
        \module{F}_{\Wobs}^\prime
        \tensor[\algebra{B}_{\Wobs}]
        \module{E}_{\Null}^\prime
        +
        \module{F}_{\Null}^\prime
        \tensor[\algebra{B}_{\Wobs}]
        \module{E}_{\Wobs}^\prime \\
        &=
        \left(
            \module{F}^\prime
            \tensor[\algebra{B}]
            \module{E}^\prime
        \right)_{\Null}
    \end{align*}
    holds. Hence $\Psi \tensor \Phi$ is a bimodule morphism from
    $\module{F} \tensor[\algebra{B}] \module{E}$ to
    $\module{F}^\prime \tensor[\algebra{B}] \module{E}^\prime$.
\end{proof}

Note that by embedding $\CoisoBimodPair(\algebra{B}, \algebra{A})$
into $\CoisoBimodTriple(\algebra{B}, \algebra{A})$ we can also define
a bimodule $\module{F} \tensor[\algebra{B}] \module{E}$ for bimodules
over coisotropic pairs.  Putting these two lemmas together we obtain
functors
\begin{equation}
    \label{eq:TensorProductFunctorTriples}
    \tensor[\algebra{B}]\colon
    \CoisoBimodTriple{(\algebra{C}, \algebra{B})}
    \times
    \CoisoBimodTriple{(\algebra{B}, \algebra{A})}
    \longrightarrow
    \CoisoBimodTriple{(\algebra{C}, \algebra{A})}.
\end{equation}
and
\begin{equation}
    \label{eq:TensorProductFunctorPairs}
    \tensor[\algebra{B}]\colon
    \CoisoBimodPair{(\algebra{C}, \algebra{B})}
    \times
    \CoisoBimodPair{(\algebra{B}, \algebra{A})}
    \longrightarrow
    \CoisoBimodPair{(\algebra{C}, \algebra{A})}.
\end{equation}
as wanted.

As a second step we need to show that the tensor product fulfills the
associativity and identity properties of a bicategory.
\begin{lemma}
    \label{lemma:AssociativityTensorProduct}%
    For coisotropic triples
    $\algebra{A}, \algebra{B}, \algebra{C}, \algebra{D}$ of algebras
    over $\mathbb{k}$ there is a natural isomorphism
    \begin{equation}
        \label{eq:Asso}
        \asso\colon
        \tensor[\algebra{B}] \circ\mathop{} (\tensor[\algebra{C}] \times \id)
        \Longrightarrow
        \tensor[\algebra{C}] \circ (\mathord{\id} \times \tensor[\algebra{B}]),
    \end{equation}
    given by the natural isomorphisms of associativity for usual
    bimodules
    \begin{equation}
        \label{eq:AssoTot}
        \asso_\Total\colon
        \left(
            \module{G}_{\Total}
            \tensor[\algebra{C}_{\Total}]
            \module{F}_{\Total}
        \right)
        \tensor[\algebra{B}_{\Total}]
        \module{E}_{\Total}
        \ni (z \tensor y) \tensor x
        \longmapsto
        z \tensor (y \tensor x) \in
        \module{G}_{\Total}
        \tensor[\algebra{C}_{\Total}]
        \left(
            \module{F}_{\Total}
            \tensor[\algebra{B}_{\Total}]
            \module{E}_{\Total}
        \right)
    \end{equation}
    and
    \begin{equation}
        \label{eq:AssoN}
        \asso_\Wobs\colon
        \left(
            \module{G}_{\Wobs}
            \tensor[\algebra{C}_{\Wobs}]
            \module{F}_{\Wobs}
        \right)
        \tensor[\algebra{B}_{\Wobs}]
        \module{E}_{\Wobs}
        \ni (z \tensor y) \tensor x
        \longmapsto
        z \tensor (y \tensor x) \in
        \module{G}_{\Wobs}
        \tensor[\algebra{C}_{\Wobs}]
        \left(
            \module{F}_{\Wobs}
            \tensor[\algebra{B}_{\Wobs}]
            \module{E}_{\Wobs}
        \right)
    \end{equation}
    for $\module{G} \in \CoisoBimodTriple{(\algebra{D},\algebra{C})}$,
    $\module{F} \in \CoisoBimodTriple{(\algebra{C},\algebra{B})}$, and
    $\module{E} \in \CoisoBimodTriple{(\algebra{B}, \algebra{A})}$.
\end{lemma}
\begin{proof}
    First it is clear that the above definitions have the necessary
    multilinearity to extend to the tensor products at all.  We then
    need to check that
    $\asso_\Total \circ \operatorname{}((\iota_\module{G} \tensor \iota_\module{F})
    \tensor \iota_\module{E}) = (\iota_\module{G} \tensor
    (\iota_\module{F} \tensor \iota_\module{E})) \circ \asso_\Wobs$
    holds and that the morphisms $\asso_\Wobs$ preserve the
    submodules. This is an easy computation done on factorizing
    tensors.
\end{proof}
\begin{lemma}
    \label{lemma:IdentitiesTensorProduct}%
    For coisotropic triples of algebras $\algebra{A}$ and
    $\algebra{B}$ over $\mathbb{k}$ there are natural isomorphisms
    \begin{equation}
        \label{eq:LeftRightIdentities}
        \lidentity\colon
        \tensor[\algebra{B}] \circ\mathop{}
        (\mathord{\Id}_\algebra{B} \times \id)
        \Longrightarrow
        \id
	\quad
        \textrm{and}
        \quad
        \ridentity\colon
        \tensor[\algebra{A}] \circ\mathop{}
        (\id \times \Id_\algebra{A})
        \Longrightarrow
        \id,
    \end{equation}
    given by the left and right identities of the tensor product of
    usual bimodules
    \begin{equation}
        \label{eq:LeftIdentitiesTotalWobs}
        \lidentity_\Total \colon
        \algebra{B}_{\Total}
        \tensor[\algebra{B}_{\Total}]
        \module{E}_{\Total}
        \ni
        b \tensor x
        \longmapsto
        bx
        \in
        \module{E}_\Total
        \quad
        \textrm{and}
        \quad
        \lidentity_\Wobs\colon
        \algebra{B}_{\Wobs}
        \tensor[\algebra{B}_{\Wobs}]
        \module{E}_{\Wobs}
        \ni
        b \tensor x
        \longmapsto
        bx
        \in
        \module{E}_\Wobs
    \end{equation}
    as well as
    \begin{equation}
        \label{eq:RightIdentitesTotalWobs}
        \ridentity_\Total\colon
        \module{E}_{\Total}
        \tensor[\algebra{A}_{\Total}]
        \algebra{A}_{\Total}
        \ni
        x \tensor a
        \longmapsto
        xa
        \in
        \module{E}_\Total
        \quad
        \textrm{and}
        \quad
        \ridentity_\Wobs\colon
        \module{E}_{\Wobs}
        \tensor[\algebra{A}_{\Wobs}]
        \algebra{A}_{\Wobs}
        \ni
        x  \tensor a
        \longmapsto
        xa
        \in
        \module{E}_\Wobs.
    \end{equation}
\end{lemma}
\begin{proof}
    Since $\lidentity_\Total$ and $\lidentity_\Wobs$ are natural
    isomorphisms we only need to show that they form morphisms of
    bimodules over triples when put together.  For this let
    $b \in \algebra{B}_\Wobs$ and $x \in \module{E}_\Wobs$, then
    \begin{align*}
  	\left(
            \lidentity_\Total
            \circ\mathop{}
            (\iota_\algebra{B} \tensor \iota_\module{E})
        \right)
        (b \tensor x)
  	=
        b \cdot \iota_\module{E}(x)
  	=
        \iota_\module{E}(b x)
  	=
        (\iota_\module{E} \circ \lidentity_\Wobs)
        (b \tensor x)
    \end{align*}
    holds.  Additionally, observe that
    \begin{align*}
        \lidentity_\Wobs
        ((\algebra{B} \tensor[\algebra{B}] \module{E})_{\Null})
        =
        \lidentity_\Wobs
        \left(
            \algebra{B}_{\Wobs}
            \tensor[\algebra{B}_{\Wobs}]
            \module{E}_{\Wobs}
            +
            \algebra{B}_{\Null}
            \tensor[\algebra{B}_{\Wobs}]
            \module{E}_{\Wobs}
        \right)
        =
        \algebra{B}_{\Wobs} \cdot \module{E}_{\Null}
        +
        \algebra{B}_{\Null} \cdot \module{E}_{\Wobs}
        =
        \module{E}_{\Null},
    \end{align*}
    since $\algebra{B}_\Wobs$ is unital and
    $\algebra{B}_\Null \cdot \module{E}_\Wobs \subseteq
    \module{E}_\Null$.
    Hence $\lidentity$ is a natural isomorphism as claimed.  An
    analogous computation shows that also $\ridentity$ is a natural
    isomorphism.
\end{proof}

Finally, Lemmas~\ref{lemma:TensorProductBimodules},
\ref{lemma:TensorProductBimoduleMorphisms},
\ref{lemma:AssociativityTensorProduct}, and
\ref{lemma:IdentitiesTensorProduct} imply that $\CoisoBimodTriple$ and
$\CoisoBimodPair$ in fact form bicategories.
\begin{theorem}[Bicategory of modules over coisotropic triples of algebras]
    \label{theorem:BicategoryModulesCoisotropicAlgebras}%
    Taking co\-iso\-tro\-pic triples of algebras as 0-morphisms,
    bimodules over coisotropic triples of algebras as 1-morphisms and
    morphisms between such bimodules as 2-morphisms, together with the
    tensor product, associativity and identities as constructed in
    Lemmas~\ref{lemma:TensorProductBimodules},
    \ref{lemma:TensorProductBimoduleMorphisms},
    \ref{lemma:AssociativityTensorProduct}, and
    \ref{lemma:IdentitiesTensorProduct} we obtain a bicategory.
\end{theorem}
\begin{proof}
    For any two coisotropic triples of algebras $\algebra{A}$ and
    $\algebra{B}$ over $\mathbb{k}$ we have a category
    $\CoisoBimodTriple(\algebra{B},\algebra{A})$ by
    Definition~\ref{definition:CoisoPairsBimodules}.  The tensor
    product as introduced in Lemma~\ref{lemma:TensorProductBimodules}
    is functorial due to
    Lemma~\ref{lemma:TensorProductBimoduleMorphisms}.  Moreover,
    Lemma~\ref{lemma:AssociativityTensorProduct} and
    Lemma~\ref{lemma:IdentitiesTensorProduct} ensure the existence of
    natural transformations of associativity and identity.  Finally,
    the coherences have to be checked, but since the associativity and
    identity natural transformations are for every component those of
    usual bimodules it is clear that the coherence diagrams from
    Definition~\ref{def:bicategory} commute.
\end{proof}

Again, by forgetting the $\TOTAL$-component we obtain also a
bicategory of coisotropic pairs, modules over coisotropic pairs and
morphisms between them:
\begin{definition}[Bicategory of modules over coisotropic algebras]
    \label{definition:BicategoryModulesCoisotropicAlgebras}
    The bicategory with co\-iso\-tropic triples of algebras as
    0-morphisms, bimodules over coisotropic triples of algebras as
    1-morphisms and morphisms between such bimodules as 2-morphisms
    from Theorem~\ref{theorem:BicategoryModulesCoisotropicAlgebras} is
    called the bicategory of coisotropic triples and will be denoted
    by $\CoisoBimodTriple$.  Similarly, the bicategory of coisotropic
    pairs of algebras, bimodules over coisotropic pairs, and bimodule
    morphisms is called the bicategory of coisotropic pairs and will
    be denoted by $\CoisoBimodPair$.
\end{definition}

We can embed the category $\Algebras$ of unital algebras into the
bicategory $\Bimodules$ of unital Algebras with bimodules by turning
an algebra morphism
$\phi \colon \algebra{A} \longrightarrow \algebra{B}$ into a bimodule
$\decorate*[_{\algebra{B}}]{\algebra{B}}{^{\phi}_{\algebra{A}}}$, with
right multiplication twisted by $\phi$, and adding as 2-morphisms only
identities.  Similarly, we can view the category $\CoisoAlgTriple$ of
coisotropic triples of algebras as a bicategory, by simply adding as
2-morphisms only the identities. This allows us to embed
$\CoisoAlgTriple$ into the bicategory $\CoisoBimodTriple$.
\begin{proposition}
    \label{propostion:EmbeddingCoisoAlgInCoisoBimod}%
    The following data defines a functor of bicategories
    $\mathsf{L} \colon \CoisoAlgTriple \longrightarrow \CoisoBimodTriple$:
    \begin{propositionlist}
    \item For $\algebra{A} \in \CoisoAlgTriple$ define
        $\functor{L}(\algebra{A}) = \algebra{A}$.
    \item For $\algebra{A}, \algebra{B} \in \CoisoAlgTriple$ a functor
        $\functor{L} \colon \CoisoAlgTriple(\algebra{B} , \algebra{A})
        \longrightarrow \CoisoBimodTriple(\algebra{B} , \algebra{A})$
        defined by
        \begin{equation}
            \label{eq:EmbeddingCoisoAlgInCoisoBimodFunctor}
            \functor{L}_{\algebra{B}\algebra{A}} (\phi)
            =
            \left(
                \deco
                {}
                {\algebra{B}_{\Total}}
                {{\left(\algebra{B}_\Total\right)}}
                {\phi}
                {\algebra{A}_{\Total}},
                \deco
                {}
                {\algebra{B}_{\Wobs}}
                {{\left(\algebra{B}_{\Wobs}\right)}}
                {\phi}
                {\algebra{A}_{\Wobs}},
                \deco
                {}
                {\algebra{B}_{\Wobs}}
                {{\left(\algebra{B}_{\Null}\right)}}
                {\phi}
                {\algebra{A}_{\Wobs}}
            \right)
        \end{equation}
        for $\phi \in \CoisoAlgTriple(\algebra{B}, \algebra{A})$.
        Here $\phi$ in superscript means that the right module
        structure is twisted with $\phi$.
    \item For
        $\algebra{A} , \algebra{B}, \algebra{C} \in \CoisoAlgTriple$ a
        natural isomorphism
        $\mathsf{m} \colon \tensor[\algebra{B}] \circ
        \operatorname{(\mathsf{L}_{\algebra{C}\algebra{B}} \times
          \functor{L}_{\algebra{B}\algebra{A}})} \longrightarrow
        \functor{L}_{\algebra{C}\algebra{A}} \circ
        \operatorname{(\mathord{\circ}_{\CoisoAlgPair})}$ given by
        \begin{equation}
            \label{eq:EmbeddingCoisoAlgInCoisoBimodMutliplication}
            \mathsf{m}(\psi, \phi)
            \colon
            \deco{}{\algebra{C}}{\algebra{C}}{\psi}{\algebra{B}}
            \tensor
            \deco{}{\algebra{B}}{\algebra{B}}{\phi}{\algebra{A}}
            \ni
            c \tensor b
            \longmapsto
            c \psi(b)
            \in
            \deco
            {}
            {\algebra{C}}
            {\algebra{C}}
            {\psi \circ \phi}
            {\algebra{A}}
        \end{equation}
        for $\phi \in \CoisoAlgTriple(\algebra{B} , \algebra{A})$ and
        $\psi \in \CoisoAlgTriple(\algebra{C},\algebra{B})$.
    \end{propositionlist}
    This functor is even an embedding of bicategories.
\end{proposition}
\begin{proof}
    The map \eqref{eq:EmbeddingCoisoAlgInCoisoBimodFunctor} defines a
    functor by the usual extension to the discrete category.  Since
    for modules of the form
    \eqref{eq:EmbeddingCoisoAlgInCoisoBimodFunctor} we have
    $\algebra{B}_\Wobs^\phi \subseteq \algebra{B}_\Total^\phi$, we
    only need to check that the natural transformation
    \eqref{eq:EmbeddingCoisoAlgInCoisoBimodMutliplication} preserves
    the $\WOBS$- and $\NULL$-components: for
    $c \tensor b \in \algebra{C}^\psi_\Wobs \tensor
    \algebra{B}^\phi_\Wobs$
    we have
    $\mathsf{m}(\psi,\phi)(c \tensor b) = c \psi(b) \in
    \algebra{C}_\Wobs$
    and for
    $c \tensor b_0 + c_0 \tensor b \in \left( \algebra{C}^\psi
        \tensor[\algebra{B}] \algebra{B}^\phi \right)_\Null =
    \algebra{C}^{\psi}_\Wobs \tensor \algebra{B}_{\Null}^{\phi} +
    \algebra{C}_{\Null}^{\psi} \tensor \algebra{B}^{\phi}_\Wobs$
    we have
    $\mathsf{m}(\psi,\phi)(c \tensor b_0 +c_0 \tensor b) = c\psi(b_0)
    + c_0\psi(b) \in \algebra{C}_{\Null}$.
    The composition and identity coherences from
    Definition~\ref{definition:FunctorBicategories} are easy
    computations.  Finally, $\functor{L}$ is clearly injective on
    objects and also on 1-morphisms, since changing
    $\phi \in \CoisoAlgTriple(\algebra{B},\algebra{A})$ will lead to
    different bimodule structures on
    $\functor{L}_{\algebra{B}\algebra{A}}(\phi)$.
\end{proof}

This embedding of $\CoisoAlgTriple$ in $\CoisoBimodTriple$ can also be
defined by omitting the $\TOTAL$-component, giving an embedding of
$\CoisoAlgPair$ in $\CoisoAlgTriple$.
\begin{remark}
    \label{remark:TOTALWOBSFunctor}%
    Note that also the projections onto the $\TOTAL$-component and the
    $\WOBS$-component yield functors of bicategories
    \begin{equation}
        \label{eq:TOTAL}
        \functor{tot}\colon
        \CoisoBimodTriple \longrightarrow \Bimodules
    \end{equation}
    and
    \begin{equation}
        \label{eq:WOBS}
        \functor{N}\colon
        \CoisoBimodTriple \longrightarrow \Bimodules,
    \end{equation}
    respectively. Note that for these functors the natural
    isomorphisms from Definition~\ref{definition:FunctorBicategories}
    are in fact identities, simplifying the situation.
\end{remark}

%
%

\section{Morita Equivalence of Coisotropic Algebras}
\label{sec:MoritaEquivalence}

In classical Morita theory two algebras (or rings) are Morita
equivalent if and only if they are isomorphic in the bicategory
$\Bimodules$ of algebras, bimodules and bimodule morphisms.  Recall
that two objects in a bicategory are called \emph{isomorphic} if there
exists an (up to 2-morphisms) invertible 1-morphism between them.
Having defined the bicategories $\CoisoBimodTriple$ and
$\CoisoBimodPair$ we can now give a definition of Morita equivalence
of coisotropic triples and pairs of algebras.
\begin{definition}[Morita equivalence of coisotropic algebras]
    \label{def: MoritaEquivalence}%
    Two coisotropic triples $\algebra{A}$, $\algebra{B}$ of algebras
    are called Morita equivalent if they are isomorphic in the
    bicategory $\CoisoBimodTriple$.  Similarly, two coisotropic pairs
    $\algebra{A}$, $\algebra{B}$ are called Morita equivalent if they
    are isomorphic in the bicategory $\CoisoBimodPair$.  An invertible
    bimodule $\algbimodule{B}{E}{A}$ implementing a Morita equivalence
    of $\algebra{A}$ and $\algebra{B}$ is called Morita equivalence
    bimodule in both cases.
\end{definition}

Thus Morita equivalence of coisotropic algebras is completely encoded
in the so called Picard bigroupoids $\Pic(\CoisoBimodTriple)$ and
$\Pic(\CoisoBimodPair)$, given by all coisotropic algebras and all
corresponding invertible 1- and 2-morphisms.

Let $\algebra{A}, \algebra{B} \in \CoisoAlgTriple$ be Morita
equivalent, and let furthermore
$\module{E} \in \CoisoBimodTriple(\algebra{B},\algebra{A})$ be a
$(\algebra{B},\algebra{A})$-bimodule implementing Morita equivalence
of $\algebra{A}$ and $\algebra{B}$.  Thus we assume that there exists
a $(\algebra{A}, \algebra{B})$-bimodule
$\module{E}^\prime \in \CoisoBimodTriple(\algebra{A},\algebra{B})$ and
isomorphisms
\begin{equation}
    \label{eq:MoritaEquivalenceIso}
    \phi\colon
    \module{E}^\prime \tensor[\algebra{B}] \module{E}
    \longrightarrow
    \algebra{A}
    \quad
    \textrm{and}
    \quad
    \psi\colon
    \module{E} \tensor[\algebra{B}] \module{E}^\prime
    \longrightarrow
    \algebra{B}
\end{equation}
such that
\begin{equation}
    \label{eq:MoritaEquivalenceIsoCompatibility}
    \psi(x \tensor x^\prime) \cdot y
    =
    x \cdot \phi( x^\prime \tensor y)
\end{equation}
holds for all $x, y \in \module{E}$ and
$x^\prime \in \module{E}^\prime$.  Note that
\eqref{eq:MoritaEquivalenceIsoCompatibility} can always be achieved by
turning a usual equivalence into an adjoint equivalence, see
\cite[A.1.3]{gurski:2007a}. The fact that
\eqref{eq:MoritaEquivalenceIso} are morphisms of coisotropic bimodules
means in particular that the diagram
\begin{equation}
    \label{eq:TensorProductStaysInjective}
    \begin{tikzcd}[column sep = large, row sep = large]
        \module{E}_\Total \tensor \module{E}_\Total^\prime
        \arrow{r}{\psi_\Total}
        & \algebra{B}_\Total \\
        \module{E}_\Wobs \tensor \module{E}_\Wobs^\prime
        \arrow{r}{\psi_\Wobs}
        \arrow{u}{\iota_\module{E} \tensor \iota_{\module{E}^\prime}}
        & \algebra{B}_\Wobs
        \arrow[swap,hookrightarrow]{u}{\iota_\algebra{B}}
    \end{tikzcd}
\end{equation}
commutes.  Then, since $\psi_\Wobs$ and $\iota_\algebra{B}$ are
injective, so is
$\psi_\Total \circ \iota_\module{E} \tensor \iota_{\module{E}^\prime}
= \iota_\algebra{B} \circ \psi_\Wobs$
and hence also $\iota_\module{E} \tensor \iota_{\module{E}^\prime}$.

Since by projecting onto the $\TOTAL$- and $\WOBS$-components yields
functors of bicategories
$\functor{tot}\colon \CoisoBimodTriple \longrightarrow \Bimodules$
and $\functor{N}\colon \CoisoBimodTriple \longrightarrow \Bimodules$
to the bicategory of algebras and bimodules according to
Remark~\ref{remark:TOTALWOBSFunctor} we know that Morita equivalence
bimodules get mapped to Morita equivalence bimodules.  Hence, by the
classical theory of Morita equivalence for algebras we know, in
particular, that $\module{E}_\Total$ and $\module{E}_\Wobs$ are
finitely generated projective modules, thus we have
\begin{align}
    \label{eq:MoritaEquivalenceFGPTotal}
    \module{E}_\Total
    &\simeq
    \mathsf{e}_\Total \algebra{A}_\Total^n
    \\
    \shortintertext{and}
    \label{eq:MoritaEquivalenceFGPWobs}
    \module{E}_\Wobs
    &\simeq
    \mathsf{e}_\Wobs \algebra{A}_\Wobs^m
\end{align}
for some $n, m \in \mathbb{N}$ and idempotents
$\mathsf{e}_\Total \in \End(\algebra{A}_\Total^n)$ and
$\mathsf{e}_\Wobs \in \End(\algebra{A}_\Wobs^m)$.  The following lemma
gives a way to relate these two finitely generated projective modules:
\begin{lemma}
    \label{lemma:DualBasisMoritaEquivalenceBimodule}%
    Let $\module{E} \in \CoisoBimodTriple(\algebra{B},\algebra{A})$ be
    a Morita equivalence bimodule of coisotropic algebras
    $\algebra{A}, \algebra{B} \in \CoisoAlgTriple$.  Then every dual
    basis $\{e_j, e^j\}_{j = 1, \ldots, n}$ of the finitely generated
    projective module $\module{E}_\Wobs$ gives rise to a dual basis
    $\{e_j^\Total, e^j_\Total\}_{j = 1, \ldots, m}$ for
    $\module{E}_\Total$, given by
    \begin{equation}
        \label{eq:DualBasisMoritaEquivalenceBimoduleBasis}
	e_j^\Total = \iota_\module{E}(e_j)
    \end{equation}
    and
    \begin{equation}
	\label{eq:DualBasisMoritaEquivalenceBimoduleDualBasis}
	e^j_\Total(x)
        =
	\bigg(
            \sum_{i=1}^{k}
            (\iota_\algebra{A} \circ e^j) (x_i^{\Unit_\Wobs})
            \cdot
            \phi_\Total
            \Big(
                \iota_{\module{E}^\prime}(y_i^{\Unit_\Wobs}) \tensor x
            \Big)
        \bigg),
    \end{equation}
    where $\Unit_{\algebra{B}_\Wobs} = \psi_\Wobs \left(
        \sum_{i=1}^{k} x_i^{\Unit_\Wobs} \tensor y_i^{\Unit_\Wobs}
    \right)$.  For $x_\Wobs \in \module{E}_\Wobs$ the dual basis
    \eqref{eq:DualBasisMoritaEquivalenceBimoduleDualBasis}
    simplifies to
    \begin{equation}
        \label{eq:DualBasisOnEWobs}
	e^j_\Total(\iota_\module{E}(x_\Wobs))
        =
        \iota_\algebra{A}\left(e^j(x_\Wobs)\right).
    \end{equation}
\end{lemma}
\begin{proof}
    First we note that we actually find elements
    $x_i^{\Unit_\Wobs} \in \module{E}_\Wobs$ and
    $y_i^{\Unit_\Wobs} \in \module{E}'_\Wobs$ such that
    $\Unit_{\algebra{B}_\Wobs} = \psi_\Wobs \big( \sum_{i=1}^{k}
    x_i^{\Unit_\Wobs} \tensor y_i^{\Unit_\Wobs} \big)$
    since $\psi_\Wobs$ is surjective.  Since
    $\algebra{B}_\Wobs \subseteq \algebra{B}_\Total$ is a unital
    subalgebra it follows that
    $\Unit_{\algebra{B}_\Total} = \psi_\Total \big( \sum_{i=1}^{k}
    \iota_\module{E}(x_i^{\Unit_\Wobs}) \tensor
    \iota_{\module{E}^\prime}(y_i^{\Unit_\Wobs}) \big) $
    by using the commutativity of
    \eqref{eq:TensorProductStaysInjective}.  Now fix a dual basis of
    $\module{E}_\Wobs$ such that for any
    $x_\Wobs \in \module{E}_\Wobs$ we have
    $x_\Wobs = \sum_{j=1}^{m} e_j \cdot e^j(x)$.  Then for
    $x \in \module{E}_\Total$ we get
    \begin{align*}
	x
	&=
        \Unit_{\algebra{B}_\Total} \cdot x \\
        &=
        \sum_{i=1}^{k} \psi_\Total
	\left(
            \iota_\module{E}(x_i^{\Unit_\Wobs})
            \tensor
            \iota_{\module{E}^\prime}(y_i^{\Unit_\Wobs})
	\right) \cdot x \\
	&\overset{(a)}{=}
        \sum_{i=1}^{k}
        \iota_\module{E}(x_i^{\Unit_\Wobs})
	\cdot
        \phi_\Total
        \left(
            \iota_{\module{E}^\prime}(y_i^{\Unit_\Wobs}) \tensor x
        \right) \\
	&=
        \sum_{i=1}^{k}
	\iota_\module{E}
        \bigg(
            \sum_{j=1}^{m}
            e_j
            \cdot
            e^j \left(x_i^{\Unit_\Wobs}\right)
        \bigg)
	\cdot
        \phi_\Total
        \left(
            \iota_{\module{E}^\prime}(y_i^{\Unit_\Wobs}) \tensor x
        \right) \\
	&=
        \sum_{i=1}^{k}
        \bigg(
            \sum_{j=1}^{m}
            \iota_\module{E}(e_j)
            \cdot
            (\iota_\algebra{A} \circ e^j)
            \left(x_i^{\Unit_\Wobs}\right)
        \bigg)
	\cdot
        \phi_\Total
        \left(
            \iota_{\module{E}^\prime}(y_i^{\Unit_\Wobs})
            \tensor
            x
        \right) \\
	&=
        \sum_{j=1}^{m}
	\underbrace{\iota_\module{E}(e_j)}_{=: e_j^\Total}
        \cdot
	\underbrace{\bigg(
              \sum_{i=1}^{k}
              (\iota_\algebra{A} \circ e^j)
              (x_i^{\Unit_\Wobs})
              \cdot
              \phi_\Total
              \left(
                  \iota_{\module{E}^\prime}(y_i^{\Unit_\Wobs})
                  \tensor
                  x
              \right)
          \bigg)}_{=:e^j_\Total(x)},
    \end{align*}
    where we used \eqref{eq:MoritaEquivalenceIsoCompatibility} in
    $(a)$.  Note that indeed $e_j^\Total \in \module{E}_\Total$ and
    $e^j_\Total \in \module{E}_\Total^*$.  Now for
    $x_\Wobs \in \module{E}_\Wobs$ we compute using
    \eqref{eq:TensorProductStaysInjective} that
    $e^j_\Total(\iota_\module{E}(x_\Wobs)) =
    \iota_\algebra{A}\left(e^j(x_\Wobs)\right)$ holds.
\end{proof}

Thus we can choose the isomorphisms
\eqref{eq:MoritaEquivalenceFGPTotal} and
\eqref{eq:MoritaEquivalenceFGPWobs} such that $n = m$.  This leads us
to the next lemma, showing that in addition for a Morita equivalence
bimodule the projectors for the $\TOTAL$- and $\WOBS$-components
agree:
\begin{lemma}
    \label{lemma:SameIdempotents}%
    Let $\module{E} \in \CoisoBimodTriple(\algebra{B},\algebra{A})$ be
    a Morita equivalence bimodule of coisotropic algebras
    $\algebra{A}, \algebra{B} \in \CoisoAlgTriple$.  Then we can
    choose the isomorphisms
    $\module{E}_\Total \simeq \mathsf{e}_\Total \algebra{A}^n_\Total$
    and $\module{E}_\Wobs \simeq \mathsf{e}_\Wobs \algebra{A}^n_\Wobs$
    such that
    $\mathsf{e}_\Total = \mathsf{e}_\Wobs \in
    \Mat_n(\algebra{A}_\Wobs)$.
\end{lemma}
\begin{proof}
    First we fix a dual basis of $\module{E}_\Wobs$ with corresponding
    dual basis of $\module{E}_\Total$ according to
    Lemma~\ref{lemma:DualBasisMoritaEquivalenceBimodule}.  Then the
    components of $\mathsf{e}_\Total \in \Mat_n(\algebra{A}_\Total)$
    are given by
    \begin{equation*}
	(\mathsf{e}_\Total)_{ij}
	= e^i_{\Total}(e_j^\Total)
	= e^i_\Total(\iota_\module{E}(e_j))
	= \iota_\algebra{A}\left(e^i(e_j)\right)
	= \iota_\algebra{A}\left( (\mathsf{e}_\Wobs)_{ij} \right).
    \end{equation*}
    Thus viewing elements of $\algebra{A}_\Wobs$ as elements in
    $\algebra{A}_\Total$ via the embedding $\iota_\algebra{A}$ gives
    the statement.
\end{proof}

As a first consequence, given a classical Morita equivalence bimodule
$\module{E}_\Total = \mathsf{e}_\Total \algebra{A}_\Total^n$ for the
$\TOTAL$-components $\algebra{B}_\Total$ and $\algebra{A}_\Total$ we
know that necessarily this can only be a Morita equivalence bimodule
for coisotropic triples of algebras
$\algebra{A} = (\algebra{A}_\Total, \algebra{A}_\Wobs,
\algebra{A}_\Null)$
and
$\algebra{B} = (\algebra{B}_\Total, \algebra{B}_\Wobs,
\algebra{B}_\Null)$
if $\mathsf{e}_\Total \in \Mat_n(\algebra{A}_\Wobs)$.

Now that we clarified the structure and relation of the $\WOBS$- and
$\TOTAL$-components of coisotropic Morita equivalence bimodules let us
turn to a description of the $\NULL$-part.  Here by definition of a
coisotropic module we have
$\module{E}_\Wobs \cdot \algebra{A}_\Null \subseteq \module{E}_\Null$.
For Morita equivalence bimodules we get in fact equality.
\begin{lemma}
    \label{lemma:ZeroComponentMoritaEquivalenceBimodules}%
    Let $\module{E} \in \CoisoBimodTriple(\algebra{B},\algebra{A})$ be
    a Morita equivalence bimodule of coisotropic algebras
    $\algebra{A}, \algebra{B} \in \CoisoAlgTriple$.  Then
    $\module{E}_\Wobs \cdot \algebra{A}_\Null = \module{E}_\Null$.
\end{lemma}
\begin{proof}
    Note that the inclusion
    $\module{E}_\Wobs \cdot \algebra{A}_\Null \subseteq \module{E}_0$
    holds by definition.  For the other inclusion we use
    $\algebra{A}_\Null \simeq (\module{E}^\prime \tensor[\algebra{B}]
    \module{E})_\Null = \module{E}^\prime_\Wobs
    \tensor[\algebra{B}_\Wobs] \module{E}_\Null +
    \module{E}^\prime_\Null \tensor[\algebra{B}_\Wobs]
    \module{E}_\Wobs$
    by \eqref{eq:MoritaEquivalenceIso} and the definition of the
    tensor product.  Thus we get
    \begin{equation*}
	\module{E}_\Wobs \cdot \algebra{A}_0
	\simeq
        \module{E}_\Wobs
        \tensor
        \module{E}^\prime_\Null
        \tensor
        \module{E}_\Wobs
	+
        \module{E}_\Wobs
        \tensor
        \module{E}^\prime_\Wobs
        \tensor
        \module{E}_\Null
	\simeq
        \module{E}_\Wobs
        \tensor
        \module{E}^\prime_\Null
        \tensor
        \module{E}_\Wobs
        +
	\module{E}_\Null
    \end{equation*}
    showing that $\module{E}_\Null \subseteq \module{E}_\Wobs \cdot
    \algebra{A}_\Null$.
\end{proof}

Putting these previous statements together we get a quite explicit
description of Morita equivalence bimodules for coisotropic algebras,
similar to the well-known description of Morita equivalence bimodules
for classical rings by Morita's theorems.
\begin{theorem}
    \label{theorem:MEBimoduleStructure}%
    Let $\module{E} \in \CoisoBimodTriple(\algebra{B},\algebra{A})$ be
    a Morita equivalence bimodule of coisotropic triples of algebras
    $\algebra{A}$ and $\algebra{B}$.  Then there exists an isomorphism
    of coisotropic bimodules such that
    \begin{align}
	\module{E}_\Total &\simeq \mathsf{e}\algebra{A}_\Total^n, \\
	\module{E}_\Wobs & \simeq \mathsf{e} \algebra{A}_\Wobs^n, \\
	\module{E}_\Null & \simeq \mathsf{e} \algebra{A}_\Null^n
    \end{align}
    with a projection $\mathsf{e} \in \Mat_n(\algebra{A}_\Wobs)$.
    Moreover, $\algebra{B}$ is completely determined by the right
    $\algebra{A}$-module structure of $\module{E}$.  We have
    \begin{align}
    	\algebra{B}_\Total
        &\simeq
        \End_{\algebra{A}_\Total}(\module{E}_\Total),
        \label{theorem:MEBimoduleStructure_LeftMulti_1}\\
    	\algebra{B}_\Wobs
        &\simeq
        \End_{\algebra{A}_\Wobs}(\module{E}_\Wobs),
    	\label{theorem:MEBimoduleStructure_LeftMulti_2}\\
    	\algebra{B}_\Null
        &\simeq
        \Hom_{\algebra{A}_\Wobs}(\module{E}_\Wobs, \module{E}_\Null),
    	\label{theorem:MEBimoduleStructure_LeftMulti_3}
    \end{align}
    where all isomorphisms are given by left multiplication and we
    view
    $\Hom_{\algebra{A}_\Wobs}(\module{E}_\Wobs, \module{E}_\Null)$ as
    a subset of $\End_{\algebra{A}_\Wobs}(\module{E}_\Wobs)$.
\end{theorem}
\begin{proof}
    Fix a dual basis $\{e_j,e^j\}_{j=1,\ldots,n}$ for
    $\module{E}_\Wobs$ and consider the dual basis
    $\{e_j^\Total, e^j_\Total\}_{j = 1, \ldots, m}$ of
    $\module{E}_\Total$ as constructed in
    Lemma~\ref{lemma:DualBasisMoritaEquivalenceBimodule}.  These dual
    bases give rise to isomorphisms
    \begin{equation*}
        g_\Wobs \colon \module{E}_\Wobs
        \ni
        \sum_{i=1}^{n} e_i e^i(x)
        \longmapsto
        \sum_{i=1}^{n} b_i e^i(x)
        \in
        \mathsf{e}\algebra{A}_\Wobs^n,
    \end{equation*}
    where $b_i$ is the standard basis of $\algebra{A}_\Wobs^n$, and
    similarly
    $g_\Total \colon \module{E}_\Total \longrightarrow
    \mathsf{e}\algebra{A}_\Total^n$.
    A straightforward computation shows that
    $\iota_\algebra{A} \circ g_\Wobs = g_\Total \circ
    \iota_\module{E}$.
    The compatibility of $g_\Wobs$ with $\module{E}_0$ is clear by
    Lemma~\ref{lemma:ZeroComponentMoritaEquivalenceBimodules}.  Thus
    we get an isomorphism of coisotropic bimodules.  Moreover, since
    $\module{E}_\Total$ and $\module{E}_\Wobs$ are classical Morita
    equivalence bimodules of the $\TOTAL$- and $\WOBS$-components,
    respectively, we immediately get
    \eqref{theorem:MEBimoduleStructure_LeftMulti_1} and
    \eqref{theorem:MEBimoduleStructure_LeftMulti_2}.  For
    \eqref{theorem:MEBimoduleStructure_LeftMulti_3} we only need to
    show that
    $\image(\algebra{B}_0) = \Hom_{\algebra{A}_\Wobs}
    (\module{E}_\Wobs, \module{E}_\Null)$
    under left multiplication.  For this let
    $\xi \in \Hom_{\algebra{A}_\Wobs}(\module{E}_\Wobs,
    \module{E}_\Null)$
    and
    $\Unit_{\algebra{B}_\Wobs} = \psi_\Wobs \big( \sum_{i=1}^{k}
    x_i^{\Unit_\Wobs} \tensor y_i^{\Unit_\Wobs} \big)$
    as before.  Then
    $\psi_\Wobs(\xi(x_i^{\Unit_\Wobs}) \tensor y_i^{\Unit_\Wobs}) \in
    \algebra{B}_0$ and
    \begin{align*}
        \psi_\Wobs\Big(
        \sum_{i=1}^{k}\xi(x_i^{\Unit_\Wobs}) \tensor y_i^{\Unit_\Wobs}
        \Big) \cdot x
        =
        \sum_{i=1}^{k}
        \xi(x_i^{\Unit_\Wobs})
        \phi_\Wobs (y_i^{\Unit_\Wobs} \tensor x)
        =
        \xi\Big(
        \sum_{i=1}^{k}
        \psi_\Wobs (x_i^{\Unit_\Wobs} \tensor y_i^{\Unit_\Wobs}) x
        \Big)
        =
        \xi(x)
    \end{align*}
    shows that
    $\algebra{B}_0 \simeq \Hom_{\algebra{A}_\Wobs}(\module{E}_\Wobs,
    \module{E}_0)$.
\end{proof}

From this it directly follows that for an equivalence bimodule the map
from the $\WOBS$-component to the $\TOTAL$-component is in fact
injective.
\begin{corollary}
    \label{corollary:InclusionIsInjectiveForMEB}%
    Let $\module{E} \in \CoisoBimodTriple(\algebra{B},\algebra{A})$ be
    a Morita equivalence bimodule for $\algebra{A}, \algebra{B} \in
    \CoisoAlgTriple$.  Then $\iota_\module{E} \colon \module{E}_\Wobs
    \longrightarrow \module{E}_\Total$ is injective, i.e.
    $\module{E}_\Wobs \subseteq \module{E}_\Total$ is a submodule.
\end{corollary}

\begin{remark}
    \label{remark:BadNews}%
    On the one hand, the theorem gives a complete picture of how the
    equivalence bimodules for coisotropic triples of algebras look
    like. On the other hand, it is quite bad news that the
    $\WOBS$-component controls and determines the other components of
    the bimodule. It will be the one which is the most inaccessible in
    the examples of deformation quantization.
\end{remark}
\begin{example}[Standard example]
    \label{example:StandardExample}%
    From the above characterization we obtain the first standard
    example: for a coisotropic triple $\algebra{A}$ the matrices
    \begin{equation}
        \label{eq:MatricesOfTriple}
        \Mat_n(\algebra{A})
        =
        \big(
        \Mat_n(\algebra{A}_\Total),
        \Mat_n(\algebra{A}_\Wobs),
        \Mat_n(\algebra{A}_\Null)
        \big)
    \end{equation}
    form again a coisotropic triple of algebras which is now Morita
    equivalent to $\algebra{A}$ for all $n \in \mathbb{N}$. As
    equivalence bimodule we can take
    \begin{equation}
        \label{eq:StandardAn}
        \algebra{A}^n
        =
        \big(
        \algebra{A}_\Total^n,
        \algebra{A}_\Wobs^n,
        \algebra{A}_\Null^n
        \big).
    \end{equation}
\end{example}

%
%

\section{Reduction for Bimodules}
\label{sec:ReductionBimodules}

Following the idea of constructing vector bundles on the reduced
manifold by reducing a vector bundle on the manifold we started with,
we want to turn bimodules over coisotropic algebras into bimodules
over the reduced algebras.  The idea is to proceed similarly to the
algebra case and consider the quotient
$\module{E}_\Wobs \big/ \module{E}_\Null$, see Proposition
\ref{proposition:CoisotropicTripleAlgebraReduction}.  Again this
construction uses only the information of the $\WOBS$- and
$\NULL$-components.  Therefore we only consider reduction for
bimodules over coisotropic pairs.  Reduction for bimodules over
coisotropic triples is then given by first forgetting about the
$\TOTAL$-component.  This construction indeed yields bimodules over
the reduced algebras and better still it is compatible with the
bicategory structure of $\CoisoBimodPair$ in the best possible way,
i.e. we get a functor of bicategories, called \emph{reduction
  functor}:
\begin{theorem}[Reduction in $\CoisoBimodPair$]
    \label{theorem:ReductionBicCoisoBimod}%
    A functor of bicategories
    $\reduce\colon \CoisoBimodPair \longrightarrow \Bimodules$ is
    given by the following data:
    \begin{theoremlist}
    \item \label{item:RedOnObjects} A map
        $\reduce \colon \Obj(\CoisoBimodPair)
        \longrightarrow\Obj(\Bimodules)$ on objects, given by
        \begin{equation}
            \label{eq:ReductionFunctorObjects}
            \algebra{A} \longmapsto \algebra{A}_\red.
        \end{equation}
    \item \label{item:RedOnOneTwoMorphs} For any two coisotropic pairs
        of algebras $\algebra{A}$ and $\algebra{B}$ a functor
        \begin{equation}
            \label{eq:ReductionABFunctor}
            \reduce_{\algebra{B}\algebra{A}} \colon
            \CoisoBimodPair(\algebra{B},\algebra{A})
            \longrightarrow
            \CoisoBimodPair(\algebra{B}_\red, \algebra{A}_\red),
        \end{equation}
        given by
        \begin{equation}
            \label{eq:ReductionFunctorOneMorphisms}
            \module{E}_\red
            =
            \faktor{\module{E}_\Wobs}{\module{E}_\Null}
        \end{equation}
        on objects and
        \begin{equation}
            \label{eq:ReductionFunctorTwoMorphisms}
            \Phi_\red
            \colon
            \module{E}_\red
            \ni
            [x]
            \longmapsto
            [\Phi(x)]
            \in
            \module{F}_\red
        \end{equation}
        on morphisms
        $\Phi \colon \module{E} \longrightarrow\module{F}$.
    \item \label{item:ReductionMultiplication} For any three
        coisotropic pairs of algebras $\algebra{A}$, $\algebra{B}$,
        and $\algebra{C}$ a natural isomorphism
        $\mathsf{m}_{\algebra{C}\algebra{B}\algebra{A}} \colon
        \tensor[\algebra{B}_\red] \circ
        \operatorname{(\red_{\algebra{C}\algebra{B}} \times
          \red_{\algebra{B}\algebra{A}})} \Longrightarrow
        \red_{\algebra{C}\algebra{A}} \circ \tensor[\algebra{B}]$
        given by a family of maps determined by
        \begin{equation}
            \label{eq:ReductionFunctorMultiplication}
            \mathsf{m}(\module{F},\module{E})
            \colon
            \module{F}_\red \tensor[\algebra{B}_\red] \module{E}_\red
            \ni
            [y] \tensor[] [x]
            \longmapsto
            [y \tensor x]
            \in
            (\module{F} \tensor[\algebra{B}] \module{E})_\red
        \end{equation}
        with
        $\module{F} \in \CoisoBimodPair(\algebra{C},\algebra{B})$,
        $\module{E} \in \CoisoBimodPair(\algebra{B},\algebra{A})$.
    \item \label{item:ReductionIdentity} For any coisotropic pair of
        algebras $\algebra{A}$ the identity 2-isomorphism
        \begin{equation}
            \label{eq:ReductionFunctorUnit}
            \id\colon
            \deco
            {}
            {\algebra{A}_\red}
            {{\algebra{A}_\red}}
            {}
            {\algebra{A}_\red}
            \longrightarrow
            \red_{\algebra{A}\algebra{A}}
            (\deco{}{\algebra{A}}{\algebra{A}}{}{\algebra{A}}).
        \end{equation}
    \end{theoremlist}
\end{theorem}
\begin{proof}
    First note that \eqref{eq:ReductionFunctorObjects} is well-defined
    since $\algebra{A}_\Null$ is a two-sided ideal in
    $\algebra{A}_\Wobs$.  Furthermore,
    \eqref{eq:ReductionFunctorOneMorphisms} gives a well-defined
    $(\algebra{B}_\red, \algebra{A}_\red)$-bimodule by the definition
    of modules over coisotropic algebras, and
    \eqref{eq:ReductionFunctorTwoMorphisms} is well-defined since
    morphisms of modules over coisotropic algebras preserve the
    submodules.  It is a standard computation to check that
    \eqref{eq:ReductionFunctorMultiplication} is a well-defined
    natural isomorphism.  Note, that here we crucially need that
    $(\module{F} \tensor[\algebra{B}] \module{E})_\Null =
    \module{F}_{\Wobs} \tensor[\algebra{B}_\Wobs] \module{E}_{\Null} +
    \module{F}_{\Null} \tensor[\algebra{B}_\Wobs] \module{E}_\Wobs$.
    For the last part it is clear that
    $\red_{\algebra{A}\algebra{A}}
    (\deco{}{\algebra{A}}{\algebra{A}}{}{\algebra{A}}) =
    \algebra{A}_\Wobs\big/\algebra{A}_\Null =
    \deco{}{\algebra{A}_\red}{{\algebra{A}_\red}}{}{\algebra{A}_\red}$.
\end{proof}

This reduction functor for $\CoisoBimodPair$ is also compatible with
the reduction in $\CoisoAlgPair$ in the sense that
\begin{equation}
    \begin{tikzcd}[column sep=large, row sep = large]
	\CoisoAlgPair
        \arrow[hookrightarrow]{r}{\functor{L}}
        \arrow{d}[swap]{\reduce}
	&\CoisoBimodPair
        \arrow{d}{\reduce} \\
	\Algebras
        \arrow[hookrightarrow]{r}{}
	&\Bimodules
    \end{tikzcd}
\end{equation}
commutes.  Indeed, on coisotropic algebras both reduction functors are
defined the same and for a morphism
$\phi \in \CoisoAlgPair(\algebra{B},\algebra{A})$ we clearly have
$\algebra{B}_\red^{\phi_\red} = (\algebra{B}^\phi)_\red$ as sets and
the $(\algebra{B}_\red,\algebra{A}_\red)$-bimodule structures also
coincide.

Moreover, by identifying isomorphic coisotropic bimodules we can
construct the classifying categories $[\CoisoBimodPair]$ and
$[\Bimodules]$.  Since $\red \colon \CoisoBimodPair \longrightarrow\Bimodules$
is a functor of bicategories we get a well-defined functor
$\red \colon [\CoisoBimodPair] \longrightarrow[\Bimodules]$, such that
\begin{equation}
    \begin{tikzcd}
	\CoisoAlgPair
        \arrow[hookrightarrow]{r}{\functor{L}}
        \arrow{d}[swap]{\reduce}
	&\CoisoBimodPair
        \arrow{d}{\reduce}
        \arrow{r}{[ \argument ]}
	&[] [\CoisoBimodPair]
        \arrow{d}{\reduce}\\
	\Algebras
        \arrow[hookrightarrow]{r}{}
	&\Bimodules
        \arrow{r}{[ \argument ]}
	&[] [\Bimodules]
    \end{tikzcd}
\end{equation}
commutes.

Also, $\red$ maps invertible morphisms to invertible
morphisms, thus it restricts to a functor
\begin{equation}
    \label{eq:ReductionOnPicardGroupoids}
    \red \colon
    \Pic(\CoisoBimodPair) \longrightarrow \Pic(\Bimodules)
\end{equation}
between the corresponding Picard bigroupoids.  Similarly, we get a
functor
$\red \colon [\Pic(\CoisoBimodPair)] \longrightarrow [\Pic(\Bimodules)]$
between the Picard groupoids, leading to the commutative diagram
\begin{equation}
    \begin{tikzcd}
	\Pic(\CoisoBimodPair)
        \arrow{d}[swap]{\reduce}
        \arrow{r}{[ \argument ]}
	&[] [\Pic(\CoisoBimodPair)]
        \arrow{d}{\reduce}\\
	\Pic(\Bimodules)
        \arrow{r}{[ \argument ]}
	&[] [\Pic(\Bimodules)]
    \end{tikzcd}
\end{equation}
This means that reduction of coisotropic algebras preserves Morita
equivalence.

%
%

\section{Classical Limit}
\label{sec:ClassicalLimit}

In formal deformation quantization one is interested in algebras
of formal power series over
a ring $\ring{R}[[\lambda]]$,
e.g. $(\Cinfty(M)[[\lambda]], \star)$ as algebra over
$\mathbb{C}[[\lambda]]$ for a Poisson manifold $M$.
Given such an
$\ring{R}[[\lambda]]$-algebra $\qalgebra{A}$ we can construct an
$\ring{R}$-algebra, called the \emph{classical limit}, by taking the
quotient $\cl(\qalgebra{A}) = \qalgebra{A} \big/ \lambda \qalgebra{A}$.
The crucial property of $\cl$ is now that all multiples of $\lambda$
will vanish, i.e. we have $\cl(\lambda a) = 0$ for all
$a \in \qalgebra{A}$.

In the following we set the underlying ring as subscript for all
involved categories in order to distinguish coisotropic triples and
pairs of algebras over $\ring{R}[[\lambda]]$ from the ones over
$\ring{R}$.

In order to define a classical limit for a coisotropic triple
$\qalgebra{A} \in \CoisoAlgTriple_{\ring{R}[[\lambda]]}$ we can not
simply set $\cl(\qalgebra{A})_\Wobs = \cl(\qalgebra{A}_\Wobs)$, since
this would not be a subset of $\cl(\qalgebra{A}_\Total)$
directly. Instead, we have to take its image in the classical limit of
the $\TOTAL$-component, leading to the following definition for the
classical limit of a coisotropic triple:
\begin{definition}[Classical limit of coisotropic triples]
    \label{definition:ClassicalLimitCoisotropicTriples}%
    Let $\qalgebra{A}$ be a coisotropic triple over
    $\ring{R}[[\lambda]]$.  Then the coisotropic triple
    \begin{align}
        \cl(\qalgebra{A})_\Total
        &=
        \cl(\qalgebra{A}_\Total) \\
        \cl(\qalgebra{A})_\Wobs
        &=
        \faktor{\qalgebra{A}_\Wobs}
	    {(\lambda \qalgebra{A}_\Total \cap \qalgebra{A}_\Wobs)}
        \subseteq
        \cl(\qalgebra{A})_\Total \\
        \cl(\qalgebra{A})_\Null
        &=
        \faktor{\qalgebra{A}_\Null}
        {(\lambda \qalgebra{A}_\Total \cap \qalgebra{A}_\Null)}
        \subseteq
        \cl(\qalgebra{A})_\Wobs
    \end{align}
    is called the \emph{classical limit} of $\qalgebra{A}$.
\end{definition}

Note that $\cl(\qalgebra{A})_0$ is indeed a two-sided ideal in
$\cl(\qalgebra{A})_\Wobs$.
In addition to the classical limit of deformed coisotropic triples we
can also take the classical limit of morphisms of coisotropic triples.
We define for a
morphism $\deform{T}\colon \qalgebra{A} \longrightarrow\qalgebra{B}$ of
coisotropic triples the classical limit
$\cl(\deform{T}) \colon \cl(\qalgebra{A})
\longrightarrow\cl(\qalgebra{B})$
by setting
$\cl(\deform{T})(\cl(\deform{a})) = \cl(\deform{T}(\deform{a}))$ for
$\deform{a} \in \qalgebra{A}_\Total$.  This is just the map defined on
the quotient since every morphism
$\deform{T} \colon \qalgebra{A} \longrightarrow\qalgebra{B}$ maps
$\lambda \qalgebra{A}_\Total$ to $\lambda \qalgebra{B}_\Total$ by
$\ring{R}[[\lambda]]$-linearity.

Let us now check that the classical limit gives a functor from
$\CoisoAlgTriple_{\ring{R}[[\lambda]]}$ to
$\CoisoAlgTriple_{\ring{R}}$.
\begin{proposition}[Classical limit functor of coisotropic triples]
    \label{prop:ClassicalLimitFunctorCoisoAlgTriple}%
    The classical limit
    \begin{equation}
        \label{eq:ClassicalLimitAlgebraTriples}
        \cl\colon
        \CoisoAlgTriple_{\ring{R}[[\lambda]]}
        \longrightarrow
        \CoisoAlgTriple_{\ring{R}}
    \end{equation}
    given by the classical limit of coisotropic triples on objects and
    quotient maps on morphisms is a functor.
\end{proposition}
\begin{proof}
    First we know that for coisotropic
    $\qalgebra{A},\qalgebra{B} \in
    \CoisoAlgTriple_{\ring{R}[[\lambda]]}$
    and a morphism
    $\deform{T} \colon \qalgebra{A} \longrightarrow\qalgebra{B}$ the
    classical limits $\cl(\qalgebra{A})$, $\cl(\qalgebra{B})$ and
    $\cl(\deform{T})$ are again coisotropic algebras and morphisms,
    respectively.  Moreover, it is clear that
    $\cl(\id_{\qalgebra{A}}) = \id_{\cl(\qalgebra{A})}$.  Let in
    addition $\qalgebra{C} \in \CoisoAlgTriple_{\ring{R}}$ be another
    coisotropic triple and
    $\deform{S} \colon \qalgebra{B} \longrightarrow \qalgebra{C}$ a
    morphism, then
    \begin{equation*}
        \cl(\deform{S} \circ \deform{T})\big(\cl(\deform{a})\big)
        =
        \cl\big((\deform{S} \circ \deform{T})(\deform{a})\big)
        =
        \cl(\deform{S})\big(\cl(\deform{T}(a))\big)
        =
        \cl(\deform{S})\big(\cl(\deform{T})(\cl(\deform{a}))\big)
        =
        \big(\cl(\deform{S}) \circ \cl(\deform{T})\big)\big(\cl(\deform{a})\big)
    \end{equation*}
    for $\deform{a} \in \qalgebra{A}$, shows that the classical limit
    is functorial.
\end{proof}

By viewing $\CoisoAlgPair_{\ring{R}[[\lambda]]}$ as a subcategory of
$\CoisoAlgTriple_{\ring{R}[[\lambda]]}$ we can also define a classical
limit functor
$\cl \colon \CoisoAlgPair_{\ring{R}[[\lambda]]} \longrightarrow
\CoisoAlgPair_\ring{R}$.

Now let us check if this classical limit is compatible with the
reduction functor for coisotropic algebras.  Since reduction of
coisotropic triples is given by forgetting the $\TOTAL$-part and
subsequent reduction of coisotropic pairs we only consider pairs from
the start.  Thus we want to clarify if the diagram
\begin{equation}
    \label{eq:RedClassicalLimitDiagram}
    \begin{tikzcd}
	\CoisoAlgPair_{\ring{R}[[\lambda]]}
        \arrow{r}{\cl}
        \arrow{d}[swap]{\reduce}
	& \CoisoAlgPair_{\ring{R}}
        \arrow{d}{\reduce} \\
	\Algebras_{\ring{R}[[\lambda]]}
        \arrow{r}{\cl}
	& \Algebras_{\ring{R}}
    \end{tikzcd}
\end{equation}
commutes.  Recall, that commutativity of a diagram of categories and
functors means that all possible compositions between the same start
and end points are the same up to natural isomorphisms.
\begin{proposition}
    \label{prop:ClassicalLimitCommutesWithReductionAlgebras}%
    There exists a natural isomorphism
    $\eta \colon (\cl \circ \red) \Longrightarrow (\red \circ \cl)$
    given by
    \begin{equation}
        \label{eq:ClassicalLimitCommutesWithReductionAlgebras}
        \eta_{\qalgebra{A}}
        \colon
        \cl(\qalgebra{A}_\red)
        \ni
        \cl([\deform{a}])
        \longmapsto
        [\cl(\deform{a})]
        \in
        \cl(\qalgebra{A})_\red
    \end{equation}
    for $\qalgebra{A} \in \CoisoAlgPair_{\ring{R}[[\lambda]]}$.
\end{proposition}
\begin{proof}
    First, an easy computation shows that $\eta_{\qalgebra{A}}$ is
    well-defined for
    $\qalgebra{A} \in \CoisoAlgPair_{\ring{R}[[\lambda]]}$.
    Similarly one can show that
    \begin{equation*}
        \eta_{\qalgebra{A}}^{-1}
        \colon \cl(\qalgebra{A})_\red
        \ni [\cl(\deform{a})]
        \longmapsto \cl([\deform{a}])
        \in \cl(\qalgebra{A}_\red)
    \end{equation*}
    is well-defined and is an inverse of $\eta_{\qalgebra{A}}$.
    Therefore we have a family of algebra isomorphisms.  Finally for
    $\qalgebra{B} \in \CoisoAlgPair_{\ring{R}[[\lambda]]}$ and
    $\deform{T} \colon \qalgebra{A} \longrightarrow\qalgebra{B}$ we
    have
    $(\red \circ \cl)(\deform{T}) \circ \eta_{\qalgebra{A}} =
    \eta_{\qalgebra{B}} \circ (\cl \circ \red)(\deform{T})$,
    as a simple evaluation on elements shows.  Thus $\eta$ is a
    natural isomorphism.
\end{proof}

Since we are interested in Morita equivalence of deformed coisotropic
algebras and the relation to the classical limit we also need to
define a classical limit for modules over coisotropic algebras.  For a
module $\qmodule{E}$ over $\ring{R}[[\lambda]]$ we define the
classical limit by
$\cl(\qmodule{E}) = \qmodule{E}\big/\lambda \qmodule{E}$, in analogy to
the case of algebras over $\ring{R}[[\lambda]]$.  This yields a
functor of bicategories
$\cl\colon \CoisoBimodTriple_{\ring{R}[[\lambda]]} \longrightarrow
\CoisoBimodTriple_\ring{R}$.
The following two lemmas will be needed to proof this.
\begin{lemma}
    \label{lemma:ClassicalLimitForCoisotropicBimodules}%
    Let $\qalgebra{A}, \qalgebra{B} \in
    \CoisoAlgTriple_{\ring{R}[[\lambda]]}$ be coisotropic algebras.
    Then the classical limit yields a functor
    \begin{equation}
        \cl
        \colon
        \CoisoBimodTriple_{\ring{R}[[\lambda]]}(\qalgebra{B}, \qalgebra{A})
        \to
        \CoisoBimodTriple_{\ring{R}}(\cl(\qalgebra{B}), \cl(\qalgebra{A}))
    \end{equation}
    given by
    \begin{align}
        \label{eq:ClassicalLimitTotal}
        \cl (\qmodule{E})_\Total
        &=
        \cl(\qmodule{E}_\Total)
        =
        \faktor{\qmodule{E}_\Total}{\lambda \qmodule{E}_\Total} \\
        \label{eq:ClassicalLimitWobs}
        \cl(\qmodule{E})_\Wobs
        &=
        \cl(\qmodule{E}_\Wobs)
        =
        \faktor{\qmodule{E}_\Wobs}{\lambda \qmodule{E}_\Wobs} \\
        \label{eq:ClassicalLimitNull}
        \cl(\qmodule{E})_\Null
        &=
        \faktor{\qmodule{E}_\Null}
		{(\lambda \qmodule{E}_\Wobs \cap \qmodule{E}_\Null)}
        \subseteq
        \cl(\qmodule{E})_\Wobs
    \end{align}
    for objects
    $\qmodule{E} \in
    \CoisoBimodTriple_{\ring{R}[[\lambda]]}(\qalgebra{B},
    \qalgebra{A})$, and by the usual map on quotients for morphisms.
\end{lemma}
\begin{proof}
    First note that the morphism
    $\iota_{\qmodule{E}} \colon \qmodule{E}_\Wobs \longrightarrow
    \qmodule{E}_\Total$
    induces a morphism
    $\iota_{\cl(\qmodule{E})} \colon \cl(\qmodule{E})_\Wobs
    \longrightarrow \cl(\qmodule{E})_\Total$
    by the $\ring{R}[[\lambda]]$-linearity of $\iota_{\module{E}}$.
    By definition,
    $\cl(\qmodule{E})_\Null \subseteq \cl(\qmodule{E})_\Wobs$ is a
    submodule.  Moreover,
    $\cl(\qalgebra{A}_\Null) \cdot \cl(\qmodule{E})_\Wobs =
    \cl(\qalgebra{A}_\Null \cdot \qmodule{E}_\Wobs) \subseteq
    \cl(\qmodule{E}_\Null)$
    and
    $\cl(\qmodule{E}_\Wobs) \cdot \cl(\qalgebra{B}_\Null) =
    \cl(\qmodule{E}_\Wobs \cdot \qalgebra{A}_\Null) \subseteq \cl
    (\qmodule{E}_\Null)$
    hold, hence $\cl(\qmodule{E})$ is a coisotropic bimodule.
    Finally, for a morphism
    $\deform{T} \colon \qmodule{E} \longrightarrow\qmodule{E}^\prime$
    between coisotropic modules, we have
    $\cl(\deform{T}) (\cl (\qmodule{E}_\Null)) = \cl (\deform{T}
    (\qmodule{E}_\Null)) \subseteq \cl( \qmodule{E}_\Null^\prime)$.
    Thus $\cl(\deform{T})$ is a morphism indeed. Then the
    functoriality is clear.
\end{proof}

Note that in contrast to the classical limit of coisotropic algebras
we do not construct $\cl(\qmodule{E})_\Wobs$ as a submodule of
$\cl(\qmodule{E})_\Total$ which is consistent with our requirement
that we only need a morphism
$\iota_{\cl(\qmodule{E})}\colon \cl(\qmodule{E})_\Wobs \longrightarrow
\cl(\qmodule{E})_\Total$.

To make this into a functor of bicategories we also need two natural
isomorphisms taking care of the composition of 1-morphisms and
identities.
\begin{lemma}
    \label{lem:ClassicalLimitCompositionNatIso}%
    Let $\qalgebra{A}, \qalgebra{B}, \qalgebra{C} \in
    \CoisoAlgTriple_{\ring{R}[[\lambda]]}$ be coisotropic triples of
    algebras over $\ring{R}[[\lambda]]$.  Moreover, let $\qmodule{F}
    \in \CoisoBimodTriple_{\ring{R}[[\lambda]]}(\qalgebra{C},
    \qalgebra{B})$ and $\qmodule{E} \in
    \CoisoBimodTriple_{\ring{R}[[\lambda]]}(\qalgebra{B},
    \qalgebra{A})$ be coisotropic bimodules.  Then
    \begin{equation}
        \label{eq:NaturalTrafom}
        \mathsf{m}
        \colon
        \cl(\qmodule{F}) \tensor[\cl(\qalgebra{B})] \cl(\qmodule{E})
        \ni
        \cl(y) \tensor \cl(x)
        \longmapsto
        \cl (y \tensor x)
        \in
        \cl(\qmodule{F} \tensor[\qalgebra{B}] \qmodule{E})
    \end{equation}
    defines a natural isomorphism $\mathsf{m} \colon
    \tensor[\cl(\qalgebra{B})] \circ \operatorname{(\cl \times \cl)}
    \Longrightarrow \cl \circ \tensor[\qalgebra{B}]$.
\end{lemma}
\begin{proof}
    A routine check shows that $\mathsf{m}$ is a well-defined
    isomorphism on both the $\TOTAL$- and $\WOBS$-component.
    Moreover, it is a morphism of coisotropic bimodules since it
    respects the $\NULL$-component, i.e.  we have
    $\mathsf{m}\left((\cl(\qmodule{F}) \tensor[\cl(\qalgebra{B})]
        \cl(\qmodule{E}))_0\right) \subseteq \cl(\qmodule{F}
    \tensor[\qalgebra{B}] \qmodule{E})_0$, and $\mathsf{m} \circ
    (\cl(\iota_{\qmodule{F}}) \tensor \cl(\iota_{\qmodule{E}})) =
    \cl(\iota_{\cl(\qmodule{F})} \tensor \iota_{\cl(\qmodule{E})})
    \circ \mathsf{m}$ holds.  Finally, one can easily check that it is
    indeed a natural isomorphism, i.e. it holds $\mathsf{m} \circ
    (\cl(\deform{T}) \tensor \cl(\deform{S}) ) = \cl(\deform{T}
    \tensor \deform{S}) \circ \mathsf{m}$, for $\deform{T} \colon
    \qmodule{F} \longrightarrow \qmodule{F}'$ and $\deform{S} \colon
    \qmodule{E} \longrightarrow \qmodule{E}'$.
\end{proof}

Putting these lemmas together we finally get the statement we aimed
for.
\begin{theorem}[Classical limit for $\CoisoBimodTriple_{\ring{R}[[\lambda]]}$]
    \label{theorem:ClisFunctor}%
    The classical limit as constructed above is a functor of
    bicategories
    \begin{equation}
        \label{eq:ClFunctor}
        \cl
        \colon
        \CoisoBimodTriple_{\ring{R}[[\lambda]]}
        \longrightarrow
        \CoisoBimodTriple_\ring{R}.
    \end{equation}
\end{theorem}
\begin{proof}
    On coisotropic algebras we use the classical limit defined in
    Proposition~\ref{prop:ClassicalLimitFunctorCoisoAlgTriple}.  For
    any two coisotropic triples
    $\qalgebra{A},\qalgebra{B} \in
    \CoisoAlgTriple_{\ring{R}[[\lambda]]}$
    there exists a classical limit functor
    $\cl \colon \CoisoBimodTriple_{\ring{R}[[\lambda]]}
    (\qmodule{B},\qmodule{A}) \longrightarrow
    \CoisoBimodTriple_\ring{R}(\cl(\qalgebra{B}),\cl(\qalgebra{A}))$
    by Lemma~\ref{lemma:ClassicalLimitForCoisotropicBimodules}.  The
    natural isomorphisms of composition are given as in
    Lemma~\ref{lem:ClassicalLimitCompositionNatIso}.  The unit
    2-isomorphisms are just the identities
    $\mathsf{u}_{\qalgebra{A}} = \id_{\cl(\qalgebra{A})}$.  The
    coherences can then be checked on elements.
\end{proof}

Since this classical limit is a functor of bicategories it drops to a
functor of the corresponding Picard (bi-)groupoids.  Thus Morita
equivalent coisotropic algebras get mapped to Morita equivalent
coisotropic algebras.  As always we can view
$\CoisoBimodPair_{\ring{R}[[\lambda]]}$ as a sub-bicategory of
$\CoisoBimodTriple_{\ring{R}[[\lambda]]}$, thus giving us a classical
limit functor for coisotropic pairs $\cl \colon
\CoisoBimodPair_{\ring{R}[[\lambda]]} \longrightarrow
\CoisoBimodPair_\ring{R}$ as well.

Now the question arises if this classical limit is compatible with
reduction, hence, if the diagram
\begin{equation}
    \label{eq:ClassicalLimitCommutesWithReductionBimodule}
    \begin{tikzcd}
	\CoisoBimodPair_{\ring{R}[[\lambda]]}
        \arrow{r}{\cl}
        \arrow{d}[swap]{\reduce}
	& \CoisoBimodPair_\ring{R}
        \arrow{d}{\reduce} \\
	\Bimodules_{\ring{R}[[\lambda]]}
        \arrow{r}{\cl}
	& \Bimodules_\ring{R}
    \end{tikzcd}
\end{equation}
commutes.  We only consider coisotropic pairs here, since we know that
reduction of triples is simply given by forgetting the
$\TOTAL$-component and using the reduction functor on pairs.  We have
to be careful here, since this is a diagram consisting of functors
between bicategories.  So instead of checking both compositions for
equality we should see if they are equal up to higher morphisms.  More
precisely, this means we have to find natural transformations $\mu
\colon (\cl \circ \red ) \Longrightarrow (\red \circ \cl )$ and
$\hat{\mu} \colon (\red \circ \cl ) \Longrightarrow (\cl \circ \red)$
of functors between bicategories and invertible modifications $\Gamma
\colon \hat{\mu} \circ \mu \longRrightarrow \id_{\cl \circ \red}$ and
$\hat{\Gamma} \colon \mu \circ \hat{\mu} \longRrightarrow \id_{\red
  \circ \cl}$ implementing that $\hat{\mu}$ is the inverse of $\mu$,
see Definition~\ref{definition:NaturalTransformationBicategories} and
Definition~\ref{definition:Modification}.  As a diagram we get
something like
\begin{equation}
    \label{eq:ClassicalLimitCommutesWithReductionHigher}
    \begin{tikzcd}[column sep = huge, row sep = huge]
        \CoisoBimodPair_{\ring{R}[[\lambda]]}
        \arrow{r}{\cl}
        \arrow{d}[swap]{\reduce}
        & \CoisoBimodPair_{\ring{R}}
        \arrow{d}{\reduce} \\
        \Bimodules_{\ring{R}[[\lambda]]}
        \arrow{r}[swap]{\cl}
        \arrow[Rightarrow, bend right=20, ru,
        start anchor = {[yshift = -.5ex]north east},
        end anchor = {[xshift = .5ex]south west},
        "\mu"{below,xshift = 1ex},
        ""{name=D}]
        \arrow[Rightarrow, bend right=20, from=ru,
        start anchor = {[yshift = .5ex]south west},
        end anchor = {[xshift=-.5ex]north east},
        "\hat{\mu}"{above},
        ""{name=U}]
        & \Bimodules_{\ring{R}}
        \arrow[from=U, to=D,
        "\Gamma",
        triplearrow,
        start anchor = west,
        end anchor = {[yshift = .5ex, xshift = -.5ex]east}]
        \arrow[from=U, to=D,
        thirdline,
        start anchor = west,
        end anchor = east]
    \end{tikzcd}
\end{equation}
There are quite a lot of things to check, so we start with giving some
properties that will later be combined to give the commutativity of
\eqref{eq:ClassicalLimitCommutesWithReductionBimodule}.

In Proposition~\ref{prop:ClassicalLimitFunctorCoisoAlgTriple} we
already showed that on objects the diagram
\eqref{eq:ClassicalLimitCommutesWithReductionBimodule} commutes up to
a natural isomorphisms.  Since we can interpret morphisms of
(coisotropic) algebras as (coisotropic) modules we can restate parts
of this result as follows:
\begin{lemma}
    \label{lem:TwistedIdentityOneMorphism}%
    Let $\qalgebra{A} \in \CoisoAlgPair_{\ring{R}[[\lambda]]}$.  Then
    $\mu_{\qalgebra{A}} =
    \deco{\eta_{\qalgebra{A}}^{-1}}{}{{(\cl(\qalgebra{A}_\red))}}{}{}$
    with $\eta_{\qalgebra{A}} \colon \cl(\qalgebra{A}_\red) \to
    \cl(\qalgebra{A})_\red$ given by
    \eqref{eq:ClassicalLimitCommutesWithReductionAlgebras} is an
    invertible $( \cl(\qalgebra{A})_\red,
    \cl(\qalgebra{A}_\red))$-bimodule in $\Bimodules_\ring{R}$.
\end{lemma}
Here
$\deco{\eta_{\qalgebra{A}}^{-1}}{}{{(\cl(\qalgebra{A}_\red))}}{}{}$
denotes the algebra $\cl(\qalgebra{A}_\red)$ regarded as a module over
itself with left multiplication twisted by the map
$\eta^{-1}_{\qalgebra{A}}$.  By essentially the same computations as
in Proposition~\ref{prop:ClassicalLimitFunctorCoisoAlgTriple} we
obtain a similar result for bimodules instead of algebras:
\begin{lemma}
    \label{lemma:ClassicalLimitCommutesWithReductionModules}%
    For every coisotropic $(\qalgebra{B}, \qalgebra{A})$-bimodule
    $\qmodule{E} \in
    \CoisoBimod_{\ring{R}[[\lambda]]}(\qalgebra{B},\qalgebra{A})$ the
    map
    \begin{equation}
        \label{eq:ClassicalLimitCommutesWithReductionModules}
        \eta(\qmodule{E})
        \colon
        \deco{\eta_{\qalgebra{B}}^{-1}}{}{(\cl(\qmodule{E}_\red))}{}{}
        \ni
        \cl([\textbf{x}])
        \longmapsto
        [\cl(\textbf{x})]
        \in
        \deco{}{}{(\cl(\qmodule{E})_\red)}{\eta_{\qalgebra{A}}}{}
    \end{equation}
    is a well-defined isomorphism of $(\cl(\qalgebra{B})_\red,
    \cl(\qalgebra{A}_\red))$-bimodules.
\end{lemma}
This family of 2-morphisms is in fact a natural transformation:
\begin{lemma}
    \label{lemma:NaturalIsosForRedAndCl}%
    For any two coisotropic algebras $\qalgebra{A}, \qalgebra{B} \in
    \CoisoBimodPair_0$ there is a natural isomorphism
    \begin{equation}
        \label{eq:NaturalMu}
        \mu
        \colon
        (\mu_{\qalgebra{B}})_*
        \circ
        (\cl \circ \red)_{\qalgebra{B}\qalgebra{A}}
        \Longrightarrow
        (\mu_{\qalgebra{A}})^*
        \circ
        (\red \circ \cl)_{\qalgebra{B}\qalgebra{A}}
    \end{equation}
    between the functors
    \begin{equation}
        \label{eq:MuIsNaturalTrafo}
        (\mu_{\qalgebra{B}})_*
        \circ
        (\cl \circ \red)_{\qalgebra{B}\qalgebra{A}}
        \colon
        \CoisoBimodPair_{\ring{R}[[\lambda]]}
        (\qalgebra{B}, \qalgebra{A})
        \longrightarrow
        \Bimodules_{\ring{R}}
        (\cl(\qalgebra{B})_\red, \cl(\qalgebra{A}_\red))
    \end{equation}
    and
    \begin{equation}
        \label{eq:MuIsStillNatural}
        (\mu_{\qalgebra{A}})^*
        \circ
        (\red \circ \cl)_{\qalgebra{B}\qalgebra{A}}
        \colon
        \CoisoBimodPair_{\ring{R}[[\lambda]]}
        (\qalgebra{B}, \qalgebra{A})
        \longrightarrow
        \Bimodules_{\ring{R}}
        (\cl(\qalgebra{B})_\red, \cl(\qalgebra{A}_\red)),
    \end{equation}
    given by the family
    \begin{equation}
        \label{eq:WhatIsMuAfterAll}
        \mu(\qmodule{E})
        =
        \ridentity^{-1}
        \circ
        \operatorname{\eta(\qmodule{E})}
        \circ
        \lidentity
        \colon
        \mu_{\qalgebra{B}} \tensor \cl(\qmodule{E}_\red)
        \longrightarrow
        \cl(\qmodule{E})_\red \tensor \mu_{\qalgebra{A}}.
    \end{equation}
    of 2-isomorphisms, with $\eta(\qmodule{E})$ as in
    Lemma~\ref{lemma:ClassicalLimitCommutesWithReductionModules}.
\end{lemma}
\begin{proof}
    It is left to show that
    \begin{equation*}
        \begin{tikzcd}[column sep=huge, row sep = large]
            \mu_{\qalgebra{B}} \tensor \cl(\qmodule{E}_\red)
            \arrow{r}{\id_{\mu_{\qalgebra{B}}} \tensor \cl(\phi_\red)}
            \arrow{d}[swap]{\mu(\qmodule{E})}
            & \mu_{\qalgebra{B}} \tensor \cl(\qmodule{F}_\red)
            \arrow{d}{\mu(\qmodule{F})} \\
            \cl(\qmodule{E})_\red \tensor \mu_{\qalgebra{A}}
            \arrow{r}{\cl(\phi)_\red \tensor \id_{\mu_{\qalgebra{A}}}}
            & \cl(\qmodule{F})_\red \tensor \mu_{\qalgebra{A}}
        \end{tikzcd}
    \end{equation*}
    commutes for all $\phi \colon \qalgbimodule{B}{E}{A}
    \longrightarrow\qalgbimodule{B}{F}{A}$.  This can be done by a
    simple computation on elements.
\end{proof}

With all these lemmas we get a natural transformation of functors
between bicategories.
\begin{lemma}
    \label{lemma:MuNaturalTrafoBicatSense}%
    The 1-morphisms $\mu_{\qalgebra{A}} \in
    \Bimodules_\ring{R}(\cl(\qalgebra{A})_\red, \cl(\qalgebra{A}_\red))$
    from Lemma~\ref{lem:TwistedIdentityOneMorphism} together with the
    natural isomorphisms
    \begin{equation}
        \label{eq:MuBAManyStarsInFormula}
        \mu
        \colon
        (\mu_{\qalgebra{B}})_*
        \circ
        (\cl \circ \red)_{\qalgebra{B}\qalgebra{A}}
        \Longrightarrow
        (\mu_{\qalgebra{A}})^*
        \circ
        (\red \circ \cl)_{\qalgebra{B}\qalgebra{A}}
    \end{equation}
    from Lemma~\ref{lemma:NaturalIsosForRedAndCl} form a natural
    transformation
    \begin{equation}
        \mu \colon (\cl \circ \red ) \Longrightarrow (\red \circ \cl)
    \end{equation}
    of functors between bicategories.
\end{lemma}
\begin{proof}
    The only things left to show are the coherence conditions for natural transformations between functors of bicategories, see
    Definition~\ref{definition:NaturalTransformationBicategories}.
    Again, this is a simple verification.
\end{proof}

This is not yet everything we need for
\eqref{eq:ClassicalLimitCommutesWithReductionBimodule} to commute.  We
still need to show that the natural transformation $\mu$ is
invertible.  For this we heavily use the fact that the 1-morphisms
$\mu_{\qalgebra{A}}$ are given by twisting the left multiplication of
$\id_{\cl(\qalgebra{A}_\red)}$ with the algebra isomorphism
$\eta_{\qalgebra{A}}^{-1}$.  Thus we can define
\begin{equation}
    \label{eq:HatMu}
    \hat\mu_{\qalgebra{A}}
    =
    \deco{\eta_{\qalgebra{A}}}{}{(\cl(\qalgebra{A})_\red)}{}{}
\end{equation}
in analogy to Lemma~\ref{lem:TwistedIdentityOneMorphism}.  Similarly,
\begin{equation}
    \label{eq:HatMuOfE}
    \hat\mu(\qmodule{E})
    =
    \ridentity^{-1} \circ  \operatorname{}\hat\eta(\qmodule{E}) \circ \lidentity
    \colon
    \hat\mu_{\qalgebra{B}} \tensor \cl(\qmodule{E})_\red
    \longrightarrow
    \cl(\qmodule{E}_\red) \tensor \hat\mu_{\qalgebra{A}}
\end{equation}
with
\begin{equation}
    \label{eq:HatMuOfEDef}
    \hat\eta(\qmodule{E})\colon
    \decorate*[^{\eta_{\qalgebra{B}}}]{{\cl(\qmodule{E})_\red}}{}
    \ni
    [\cl(\mathbf{x})]
    \longmapsto
    \cl([\mathbf{x}])
    \in
    \module{{\cl(\qmodule{E}_\red)}}^{\eta_{\qalgebra{B}}^{-1}}
\end{equation}
gives a natural isomorphism
\begin{equation}
    \hat\mu
    \colon
    (\hat{\mu}_{\qalgebra{B}})_* \circ (\red \circ \cl)_{\qalgebra{B}\qalgebra{A}}
    \Longrightarrow
    (\hat{\mu}_{\qalgebra{A}})^* \circ ( \cl \circ \red)_{\qalgebra{B}\qalgebra{A}}
\end{equation}
in analogy to Lemma~\ref{lemma:NaturalIsosForRedAndCl}.  This
yields again a natural transformation of functors between
bicategories.
\begin{lemma}
    \label{lemma:HatMuNaturalTrafo}%
    The 1-morphisms $\hat\mu_{\qalgebra{A}} \in
    \CoisoBimod_\ring{R}(\cl(\qalgebra{A}_\red),
    \cl(\qalgebra{A})_\red)$ together with the natural isomorphisms
    \begin{equation}
        \label{eq:HatMuForBA}
        \hat\mu
        \colon
        (\hat\mu_{\qalgebra{B}})_* \circ (\red \circ \cl)_{\qalgebra{B}\qalgebra{A}}
        \Longrightarrow
        (\hat\mu_{\qalgebra{A}})^* \circ ( \cl \circ \red )_{\qalgebra{B}\qalgebra{A}}
    \end{equation}
    form a natural transformation
    \begin{equation}
        \label{eq:HatMuNatural}
        \hat\mu \colon (\red \circ \cl ) \Longrightarrow (\cl \circ \red).
    \end{equation}
    of functors between bicategories.
\end{lemma}
Now the last thing to show is that $\mu$ and $\hat{\mu}$ are indeed
inverse to each other: this is of course to be understood in the sense
of natural transformations between bicategories and hence up to a
modification:
\begin{lemma}
    \label{lemma:GammaModification}%
    The natural transformations $\mu \colon (\cl \circ \red )
    \Longrightarrow (\red \circ \cl)$ and $\hat{\mu} \colon (\red
    \circ \cl) \Longrightarrow (\cl \circ \red)$ are inverse to each
    other.
\end{lemma}
\begin{proof}
    We need to show that there are invertible modifications
    $\Gamma \colon \hat{\mu} \circ \mu \longRrightarrow \id_{\cl \circ
      \red}$
    and
    $\hat{\Gamma} \colon \mu \circ \hat{\mu} \longRrightarrow
    \id_{\red \circ \cl}$.
    Hence, we need for any
    $\qalgebra{A} \in \CoisoBimod_{\ring{R}[[\lambda]]}$ a
    2-isomorphism
    $\Gamma_{\qalgebra{A}} \colon \hat{\mu}_{\qalgebra{A}} \tensor
    \mu_{\qalgebra{A}} \longrightarrow \id_{\cl(\qalgebra{A}_\red)}$.
    Recall that
    $\hat{\mu}_{\qalgebra{A}} =
    \decorate*[^{\eta_{\qalgebra{A}}}]{{\cl(\qalgebra{A})_\red}}{}$
    and
    $ \mu_{\qalgebra{A}} =
    \decorate*[^{\eta_{\qalgebra{A}}^{-1}}]{{\cl(\qalgebra{A}_\red)}}{}$,
    thus we get an isomorphism $\Gamma_{\qalgebra{A}}$ by
    \begin{equation*}
        \hat{\mu}_{\qalgebra{A}} \tensor \mu_{\qalgebra{A}}
        =
        \decorate*[^{\eta_{\qalgebra{A}}}]{{\cl(\qalgebra{A})_\red}}{}
        \tensor
        \decorate*[^{\eta_{\qalgebra{A}}^{-1}}]{{\cl(\qalgebra{A}_\red)}}{}
        \simeq
        \decorate*[^{(\eta_{\qalgebra{A}} \circ
          \eta_{\qalgebra{A}}^{-1})}]
        {{\cl(\qalgebra{A}_\red)}}{}
        =
        \cl(\qalgebra{A}_\red)
        =
        \id_{\cl(\qalgebra{A}_\red)},
    \end{equation*}
    mapping $[\cl(\deform{a})] \tensor \cl([\deform{b}])$ to
    $\cl([\deform{ab}])$ and with inverse mapping $\cl([\deform{a}])$
    to $[\cl(\Unit_{\qalgebra{A}})] \tensor \cl([\deform{a}])$.  With
    this isomorphism the diagram
    \begin{equation*}
        \begin{tikzcd}[column sep = large, row sep = large]
            (\hat{\mu}_{\qalgebra{B}} \tensor \mu_{\qalgebra{B}})
            \tensor
            \cl(\qmodule{E}_\red)
            \arrow{r}{(\hat{\mu} \circ \mu)(\qmodule{E})}
            \arrow{d}[swap]{\Gamma_{\qalgebra{B}} \tensor \id}
            &
            \cl(\qmodule{E}_\red)
            \tensor
            (\hat{\mu}_{\qalgebra{A}} \tensor \mu_{\qalgebra{A}})
            \arrow{d}{\id \tensor \Gamma_{\qalgebra{A}}} \\
            \id_{\cl(\qalgebra{B}_\red)} \tensor \cl(\qmodule{E}_\red)
            \arrow{r}{\id_{(\cl \circ \red)}(\qmodule{E})}
            & \cl(\qmodule{E}_\red) \tensor \id_{\cl(\qalgebra{A}_\red)}
        \end{tikzcd}
    \end{equation*}
    commutes.
    Similarly, we obtain an isomorphism $\hat{\Gamma}_{\qalgebra{A}}$ by
    \begin{equation*}
    	\mu_{\qalgebra{A}} \tensor \hat{\mu}_{\qalgebra{A}}
    	=
        \decorate*[^{\eta^{-1}_{\qalgebra{A}}}]{{\cl(\qalgebra{A}_\red)}}{}
        \tensor
        \decorate*[^{\eta_{\qalgebra{A}}}]{{\cl(\qalgebra{A})_\red}}{}
        \simeq
        \decorate*[^{\eta_{\qalgebra{A}}^{-1} \circ
          \eta_{\qalgebra{A}}}]
        {{\cl(\qalgebra{A})_\red}}{}
        =
        \cl(\qalgebra{A})_\red
        =
        \id_{\cl(\qalgebra{A})_\red},
    \end{equation*}
    mapping $\cl([\deform{a}]) \tensor[] [\cl(\deform{b})]$ to
    $[\cl(\deform{ab})]$ and inverse mapping $[\cl(\deform{a})]$ to
    $\cl([\Unit_{\qalgebra{A}}]) \tensor[] [\cl(\deform{a})]$.
\end{proof}

Thus we finally see that
\eqref{eq:ClassicalLimitCommutesWithReductionBimodule} commutes:
\begin{theorem}
    \label{theorem:ClassicalLimitCommutesWithReduction}%
    The classical limit on $\CoisoBimodPair_{\ring{R}[[\lambda]]}$
    commutes with reduction, i.e. the diagram
    \eqref{eq:ClassicalLimitCommutesWithReductionBimodule} given as
    \begin{equation}
        \label{eq:RedCommutesWithCL}
	\begin{tikzcd}
            \CoisoBimodPair_{\ring{R}[[\lambda]]}
            \arrow{r}{\cl}
            \arrow{d}[swap]{\reduce}
            & \CoisoBimodPair_\ring{R}
            \arrow{d}{\reduce} \\
            \Bimodules_{\ring{R}[[\lambda]]}
            \arrow{r}{\cl}
            & \Bimodules_\ring{R}
	\end{tikzcd}
    \end{equation}
    commutes up to the invertible natural transformations $\mu$ and
    $\hat{\mu}$.
\end{theorem}

Thinking in geometric terms the Morita equivalence on the classical
side is well-understood.  Moreover, Morita equivalence after reduction
is just the classical Morita equivalence.  Thus if we want to
understand Morita equivalence in
$\CoisoBimodPair_{\ring{R}[[\lambda]]}$ better it might be helpful to
examine the functors $\cl$ and $\red$ in order to transport knowledge
about the classical or reduced side back to
$\CoisoBimodPair_{\ring{R}[[\lambda]]}$.

A first observation is that by taking Picard (bi-)groupoids of all
involved bicategories in
Theorem~\ref{theorem:ClassicalLimitCommutesWithReduction} immediately
yields the commutativity of
\begin{equation}
    \label{eq:RedCommutesWithCLOnPic}
    \begin{tikzcd}
	\Pic(\CoisoBimodPair_{\ring{R}[[\lambda]]})
        \arrow{r}{\cl}
        \arrow{d}[swap]{\reduce}
	& \Pic(\CoisoBimodPair_\ring{R})
        \arrow{d}{\reduce} \\
	\Pic(\Bimodules_{\ring{R}[[\lambda]]})
        \arrow{r}{\cl}
	& \Pic(\Bimodules_\ring{R}).
    \end{tikzcd}
\end{equation}

%
%

\appendix

\section{Bicategories}
\label{sec:Bicategories}

For the convenience of the reader and to explain our conventions, we
collect some basic definitions concerning bicategories, see
\cite{benabou.et.al:1967a} or \cite{leinster:2004a} for a more modern
treatment.
\begin{definition}[Bicategory]
    \label{def:bicategory}%
    A \emph{bicategory} $\bicategory{B}$ consists of the following data:
    \begin{definitionlist}
    \item A class $\category{B}_0$, the objects of $\bicategory{B}$.
    \item For any two objects $A,B \in \category{B}_0$ a category
        $\bicategory{B}(B,A)$.  The objects $\category{B}_1(B,A) =
        \Obj(\bicategory{B}(B,A))$ of this category are called
        1-morphisms from $A$ to $B$.  Morphisms $\phi \colon f
        \longrightarrow g$ between 1-morphisms $f, g \in
        \category{B}_1(B,A)$ are called 2-morphisms from $f$ to $g$.
        The set of such 2-morphisms is denoted by
        $\category{B}_2(g,f)$.
    \item For any three objects $A,B,C \in \category{B}_0$ a functor
        \begin{equation}
            \tensor[CBA] \colon
            \bicategory{B}(C,B) \times \bicategory{B}(B,A)
            \longrightarrow
            \bicategory{B}(C,A),
        \end{equation}
        called the composition or tensor product of 1-morphisms.
    \item For each object $A \in \category{B}_0$ a 1-morphism $\Id_A
        \in \category{B}_1(A,A)$, called the identity at $A$.
    \item For any four objects $A,B,C,D \in \category{B}_0$ a natural
        isomorphism
        \begin{equation}
            \label{eq:BicategoryAssoNaturalIsomorphism}
            \asso_{DCBA}
            \colon
            \tensor[CBA]
            \circ\mathop{}
            \left(\tensor[DCB] \times \id \right)
            \Longrightarrow
            \tensor[DCB] \circ \left( \id \times \tensor[CBA]\right),
        \end{equation}
        called the associativity.
    \item For any two objects $A,B \in \category{B}_0$ natural
        isomorphisms
        \begin{equation}
            \label{eq:BicategoryLeftNaturalIsomorphism}
            \lidentity_{BA} \colon
            \tensor[BBA] \circ \left( \id_B \times \id \right)
            \Longrightarrow
            \id
        \end{equation}
        and
        \begin{equation}
            \label{eq:BicategoryRightNaturalIsomorphism}
            \ridentity_{BA} \colon
            \tensor[BAA] \circ \left( \id \times \id_A \right)
            \Longrightarrow
            \id,
        \end{equation}
        called the left and right identity, respectively.
    \end{definitionlist}
    These data are required to fulfill the following coherence
    conditions:
    \begin{definitionlist}
    \item Associativity coherence: the diagram
        \begin{equation}
            \label{eq:BicatgeoryAssociativityCoherence}
            \begin{tikzcd}[column sep=0.4em]
                ((k \tensor[D] h) \tensor[C] g) \tensor[b] f
                \arrow{rr}{\asso(k,h,g) \tensor[B] \id}
                \arrow{d}[swap]{\asso(k \tensor[D] h,g,f)}
                &
                & (k \tensor[D] ( h \tensor[C] g)) \tensor[B] f
                \arrow{d}{\asso(k, h \tensor[C] g, f) } \\
                (k \tensor[D] h) \tensor[C] (g \tensor[B] f)
                \arrow{dr}[swap]{\asso(k,h,g \tensor[B] f)}
                &
                & k \tensor[D] ((h \tensor[C] g) \tensor[B] f)
                \arrow{dl}{\id \tensor[D] \asso(h,g,f)} \\
                & k \tensor[D] (h \tensor[C] (g \tensor[B] f))
            \end{tikzcd}
        \end{equation}
        commutes for all $k \in \category{B}_1(E,D)$, $h \in
        \category{B}_1(D,C)$, $g \in \category{B}_1(C,B)$ and $f
        \in \category{B}_1(B,A)$.
    \item Identity coherence: the diagram
        \begin{equation}
            \label{eq:BicategoryIdentityCoherence}
            \begin{tikzcd}
                (g \tensor[B] \Id_B) \tensor[B] f
                \arrow{rr}{\asso(g, \Id_B, f)}
                \arrow{dr}[swap]{\ridentity(g) \tensor[B] \id}
                &
                & g \tensor[B] (\Id_B \tensor[B] f)
                \arrow{dl}{\id \tensor[B] \lidentity(f)} \\
                &g \tensor[B] f
                & { }
            \end{tikzcd}
        \end{equation}
        commutes for all $g \in \category{B}_1(C,B)$ and $f \in
        \category{B}_1(B,A)$.
    \end{definitionlist}
\end{definition}

Note that we simplify $\tensor[CBA]$ to $\tensor[B]$ and drop indices
of the involved natural isomorphisms whenever there is no possibility
of confusion.  Recall that in bicategories there is a way to compose
1-morphisms with 2-morphisms.  Let
\begin{equation}
    \label{eq:fggprimeDiagram}
    \begin{tikzcd}
	A
        \arrow[r, "f"]
	&B
        \arrow[r, bend left=50,"g", " "{name=U, below}]
        \arrow[r, bend right=50, "g^\prime"{below}, " "{name=D}]
	&C
	\arrow[Rightarrow, "\phi", from=U, to=D	]
    \end{tikzcd}
\end{equation}
be given, then we get a 2-morphism
$f^*\phi = \phi \tensor[B] \id_f \colon g \tensor f \longrightarrow
g^\prime \tensor f$
between the (horizontal) compositions of $f$ and $g$, and $f$ and
$g^\prime$, respectively.
In the same way, given
\begin{equation}
    \label{eq:ggprimehDiagram}
    \begin{tikzcd}
	B
        \arrow[r, bend left=50,"g", " "{name=U, below}]
        \arrow[r, bend right=50, "g^\prime"{below}, " "{name=D}]
	& C
        \arrow[Rightarrow, "\phi", from=U, to=D	]
        \arrow[r, "h"]
	& D
    \end{tikzcd}.
\end{equation}
one defines a 2-morphism
$h_*\phi = \id_h \tensor[C] \phi \colon h \tensor g \longrightarrow h
\tensor g^\prime$
between the (horizontal) compositions.
These compositions can also be seen as functors between the
appropriate hom-categories and are sometimes called \emph{whiskering}.

As morphisms of bicategories we use what is often called a weak
(2-)functor or pseudofunctor. Note that there are also weaker
notions like lax and oplax functor which, however, will not suffice
for our purposes. In the original work Benabou calls the following
version a homomorphism of bicategories \cite{benabou:1967a}:
\begin{definition}[Functor of bicategories]
    \label{definition:FunctorBicategories}%
    Let $\bicategory{A}$ and $\bicategory{B}$ be two bicategories.  A
    functor $\functor{F}$ from $\bicategory{A}$ to $\bicategory{B}$,
    written
    $\functor{F} \colon \bicategory{A} \longrightarrow
    \bicategory{B}$, consists of the following data:
    \begin{definitionlist}
    \item \label{item:FunctorNullMorphisms} A map
        $\functor{F} \colon \category{A}_0 \longrightarrow
        \category{B}_0$
        mapping objects of $\bicategory{A}$ to objects of
        $\bicategory{B}$.
    \item \label{item:FunctorOneMorphCat} For any two objects
        $A,B \in \category{A}_0$ a functor
        \begin{equation}
            \label{eq:FunctorBicatFBA}
            \functor{F}_{BA}\colon
            \bicategory{A}(B,A)
            \longrightarrow
            \bicategory{B}(\functor{F}B ,\functor{F}A).
        \end{equation}
    \item \label{item:FunctorNaturalIso} For each three objects
        $A,B,C \in \category{A}_0$ a natural isomorphism
        \begin{equation}
            \label{eq:FunctorBicategoriesTensorIsomorphism}
            \mathsf{m}_{CBA}\colon
            \tensor[\functor{F}B]
            \circ\mathop{}
            \left(\functor{F}_{CB} \times \functor{F}_{BA} \right)
            \Longrightarrow
            \functor{F}_{CA} \circ \tensor[B].
        \end{equation}
    \item \label{item:FunctorBicatTwoIso} For any object
        $A \in \category{A}_0$ a 2-isomorphism
        \begin{equation}
            \mathsf{u}_A\colon
            \Id_{\functor{F}A}
            \longrightarrow
            \functor{F}_{AA} (\Id_A).
        \end{equation}
    \end{definitionlist}
    These data are required to fulfil the following coherence
    conditions:
    \begin{definitionlist}
    \item \label{item:CompositionCoherence} Composition coherence: the
        diagram
        \begin{equation}
            \label{eq:FunctorBicategoriesComposition}
            \begin{tikzcd}
                \functor{F}h \tensor[\functor{F}C] (\functor{F}g \tensor[\functor{F}B] \functor{F}f)
                \arrow{r}{\asso}
                \arrow{d}[swap]{\id_{\functor{F}h} \tensor \mathsf{m}(g,f)}
                & (\functor{F}h \tensor[\functor{F}C] \functor{F}g ) \tensor[\functor{F}B] \functor{F}f
                \arrow{d}{\mathsf{m}(h ,g) \tensor \id_{\functor{F}f}}\\
                \functor{F}h \tensor[\functor{F}C] \functor{F}(g \tensor[B] f)
                \arrow{d}[swap]{\mathsf{m}(h , g \tensor[B] f)}
                & \functor{F}(h \tensor[C] g) \tensor[\functor{F}B] \functor{F}f
                \arrow{d}{\mathsf{m}(h \tensor[C] g , f)} \\
                \functor{F}(h \tensor[C] (g \tensor[B] f) )
                \arrow{r}{\functor{F}(\asso)}
                & \functor{F}((h \tensor[C] g) \tensor[B] f)
            \end{tikzcd}
        \end{equation}
        commutes for all $h \in \category{A}_1(D,C)$,
        $g \in \category{A}_1(C,B)$ and $f \in \category{A}_1(B,A)$.
    \item \label{item:IdentityCoherence} Identity coherence: the
        diagram
        \begin{equation}
            \label{eq:FunctorBicategoriesIdentityCoherence}
            \begin{tikzcd}[column sep = huge]
                \Id_{\functor{F}B} \tensor[\functor{F}B] \functor{F}f
                \arrow{r}{\lidentity(\functor{F}f)}
                \arrow{d}[swap]{\mathsf{u}_B \tensor[\functor{F}B] \id_{\functor{F}f}}
                & \functor{F}f
                \arrow{dd}{\id}
                & \functor{F}f \tensor[\functor{F}A] \Id_{\functor{F}A}
                \arrow{l}[swap]{\ridentity(\functor{F}f)}
                \arrow{d}{\id_{\functor{F}f} \tensor[\functor{F}A] \mathsf{u}_A}\\
                \functor{F}(\Id_B) \tensor[\functor{F}B] \functor{F}f
                \arrow{d}[swap]{\mathsf{m}(\Id_B,f)}
                & { }
                &\functor{F}f \tensor[\functor{F}A] \functor{F}(\Id_A)
                \arrow{d}{\mathsf m(f , \Id_A)}\\
                \functor{F}( \Id_B \tensor[B] f )
                \arrow{r}{\functor{F}(\lidentity(f))}
                &\functor{F}f
                & \functor{F}(f \tensor[A] \Id_A)
                \arrow{l}[swap]{\functor{F}(\ridentity(f))}
            \end{tikzcd}
        \end{equation}
        commutes for all $f \in \category{A}_1(B,A)$.
    \end{definitionlist}
\end{definition}

Composition of functors of bicategories is defined by composing the
obvious maps, functors and natural transformations.  Similar to usual
categories there is also a notion of natural transformation between
functors.  But now we have to incorporate the higher morphisms.
\begin{definition}[Natural transformation]
    \label{definition:NaturalTransformationBicategories}%
    Let
    $\functor{F}, \functor{G} \colon \bicategory{A} \longrightarrow
    \bicategory{B}$
    be functors between bicategories $\bicategory{A}$ and
    $\bicategory{B}$.  A natural transformation $\eta$ from
    $\functor{F}$ to $\functor{G}$, written
    $\eta \colon \functor{F} \Longrightarrow \functor{G}$, consists of
    the following data:
    \begin{definitionlist}
    \item \label{item:BicatNaturalTrafoOneMorph} for each
        $A \in \bicategory{A}_0$ a 1-morphism
        $\eta_A \colon \functor{F}A \longrightarrow \functor{G}A$ in
        $\bicategory{B}$.
    \item \label{item:BicatNaturalTrafoTwoMorph} for each 1-morphism
        $f \in \bicategory{A}_1(B,A)$ a 2-isomorphism
        \begin{equation}
            \label{eq:NaturalTransformationBicategoriesEta}
            \eta_f \colon
            \eta_B \tensor[\functor{F}B] \functor{F}f
            \longrightarrow \functor{G}f \tensor[\functor{G}A] \eta_A,
        \end{equation}
        such that for any $A, B \in \bicategory{A}_0$ the
        2-morphisms $\eta_f$ are the components of a natural
        isomorphism
        \begin{equation}
            \label{eq:NaturalTransformationorBicategoriesEtaProperties}
            \eta_{BA} \colon
            (\eta_B)_* \circ \functor{F}_{BA}
            \Longrightarrow
            (\eta_A)^* \circ \functor{G}_{BA} .
        \end{equation}
    \end{definitionlist}
    These data are required to fulfil the following coherence
    conditions:
    \begin{definitionlist}
    \item \label{item:NaturalTrafoBicatBigDiagram} The diagram
        \begin{equation}
            \label{eq:NaturalTransformationBicategoriesBigDiag}
            \begin{tikzcd}
                \eta_C \tensor[\functor{F}C] \left( \functor{F}g \tensor[\functor{F}B] \functor{F}f \right)
                \arrow{r}{\id \tensor \mathsf{m}^\functor{F}}
                \arrow{d}[swap]{\asso^{-1}}
                & \eta_C \tensor[\functor{F}C] \functor{F}\left( g \tensor f \right)
                \arrow{ddddd}{\eta(g \tensor f)} \\
                \left( \eta _C \tensor[\functor{F}C] \functor{F}g \right) \tensor[\functor{F}B] \functor{F}f
                \arrow{d}[swap]{\eta(g) \tensor \id}
                & { } \\
                \left( \functor{G}g \tensor[\functor{G}B] \eta_B \right) \tensor[\functor{F}B] \functor{F}f
                \arrow{d}[swap]{\asso}
                & { } \\
                \functor{G}g \tensor[\functor{G}B] \left( \eta_B \tensor[\functor{F}B] \functor{F}f \right)
                \arrow{d}[swap]{\id \tensor \eta(f)}
                & { } \\
                \functor{G}g \tensor[\functor{G}B] \left( \functor{G}f \tensor[\functor{G}A] \eta_A \right)
                \arrow{d}[swap]{\asso^{-1}}
                & { } \\
                \left( \functor{G}g \tensor[\functor{G}B] \functor{G}f \right) \tensor[\functor{G}A] \eta_A
                \arrow{r}{\mathsf{m}^\functor{G} \tensor \id}
                & \functor{G}(g \tensor f) \tensor[\functor{G}A] \eta_A
            \end{tikzcd}
        \end{equation}
        commutes for all $f \in \bicategory{A}_1(B,A)$ and
        $g \in \bicategory{A}_1(C,B)$.
    \item \label{item:NaturalTrafoBicatSmallDiagram} The diagram
        \begin{equation}
            \label{eq:NaturalTransformationBicategoriesSmallDiagram}
            \begin{tikzcd}
                \eta_A \tensor[\functor{F}A] \Id_{\functor{F}A}
                \arrow{r}{\ridentity}
                \arrow{d}[swap]{\id \tensor \mathsf{u}_A }
                & \eta_A
                \arrow{r}{\lidentity^{-1}}
                & \Id_{\functor{G}A} \tensor[\functor{G}A] \eta_A
                \arrow{d}{\mathsf{u}_A \tensor \id}\\
                \eta_A \tensor[\functor{F}A] \functor{F}(\Id_A)
                \arrow{rr}{\eta_{\Id_A}}
                & { }
                & \functor{G}(\Id_A) \tensor[\functor{G}A] \eta_A
            \end{tikzcd}
        \end{equation}
        commutes for all $A \in \bicategory{A}_0$.
    \end{definitionlist}
\end{definition}

For bicategories there is also the possibility to relate natural
transformations via so called modifications:
\begin{definition}[Modification]
    \label{definition:Modification}%
    Let $\bicategory{A}$ and $\bicategory{B}$ be bicategories.  Let
    furthermore $\eta \colon \functor{F} \Longrightarrow \functor{G}$
    and $\mu \colon \functor{F} \Longrightarrow \functor{G}$ be two
    natural transformations between functors
    $\functor{F}, \functor{G} \colon \bicategory{A} \longrightarrow
    \bicategory{B}$.
    A modification $\Gamma \colon \eta \longRrightarrow \mu$ is an
    assignment that assigns to every object $A \in \bicategory{A}_0$ a
    2-morphism $\Gamma_A \colon \eta_A \longrightarrow \mu_A$ such
    that for each morphism $f \in \bicategory{A}_1(B,A)$ the
    diagram
    \begin{equation}
	\begin{tikzcd}
            \eta_B \tensor[\functor{F}B] \functor{F}f
            \arrow{d}[swap]{\Gamma_B \tensor \id}
            \arrow{r}{\eta(f)}
            & \functor{G}f \tensor[\functor{G}A] \eta_A
            \arrow{d}{\id \tensor \Gamma_A} \\
            \mu_B \tensor[\functor{F}B] \functor{F}f
            \arrow{r}{\mu(f)}
            & \functor{G}f \tensor[\functor{G}A] \mu_A
	\end{tikzcd}
    \end{equation}
    commutes.
\end{definition}

%
%

%
%

\end{document}
